\documentclass{amsart}
\usepackage{amsmath, amsthm, amssymb, amsbsy}
\usepackage{hyperref} 
\usepackage{enumitem}
\usepackage{stmaryrd}
\usepackage{graphicx} 
\usepackage{xcolor}

\usepackage{tikz}
\usepackage{pgf}
\usepackage{tikzscale}
\usepackage{tikz-cd}

\usetikzlibrary{decorations.pathreplacing,decorations.markings}
\usetikzlibrary{arrows, knots, external, positioning}

\usepackage{caption} 
\usepackage[a4paper,width=150mm,top=25mm,bottom=25mm,centering]{geometry}

\newtheorem{theorem}{Theorem}[section]
\newtheorem{conjecture}[theorem]{Conjecture}
\newtheorem{lemma}[theorem]{Lemma}
\newtheorem{proposition}[theorem]{Proposition}
\newtheorem{corollary}[theorem]{Corollary}

\newtheorem{problem}[theorem]{Problem}

\theoremstyle{definition}
\newtheorem{definition}[theorem]{Definition}
\newtheorem{remark}[theorem]{Remark}
\newtheorem{example}[theorem]{Example}

\renewcommand{\Bbb}{\mathbb}

\begin{document}

\title{On rational homology projective planes with quotient singularities of small indices}

\author{Woohyeok Jo}
\address{Department of Mathematical Sciences, Seoul National University, Seoul 08826,
 Republic of Korea}
\email{koko1681@snu.ac.kr}

\author{Jongil Park}
\address{Department of Mathematical Sciences and Research Institute of Mathematics, Seoul National University, Seoul 08826, Republic of Korea}
\email{jipark@snu.ac.kr}

\author{Kyungbae Park}
\address{Department of Mathematics, Kangwon National University, Kangwon 24341,
 Republic of Korea}
\email{kyungbaepark@kangwon.ac.kr}

\thanks{}
\date{\today}

\begin{abstract}
In this article, we study the effects of topological and smooth obstructions on the existence of rational homology complex projective planes that admit quotient singularities of small indices. In particular, we provide a classification of the types of quotient singularities that can be realized on rational homology complex projective planes with indices up to three, whose smooth loci have trivial first integral homology group.
\end{abstract}

\maketitle

\section{Introduction}
A normal projective complex surface $S$ whose Betti numbers $b_i(S)$ are the same as those of the complex projective plane $\mathbb{CP}^2$ is called a \textit{rational homology projective plane} (or a $\mathbb{Q}$-homology $\mathbb{CP}^2$). Following the convention of \cite{Hwang-Keum-Ohashi-2015}, we assume, throughout this paper,
that a $\mathbb{Q}$-homology $\mathbb{CP}^2$ has at worst quotient singularities. 

For a $\mathbb{Q}$-homology $\mathbb{CP}^2$, either $\pm K$ is ample or $K$ is numerically trivial, where $K$ denotes the canonical divisor. A $\mathbb{Q}$-homology $\mathbb{CP}^2$ with an anti-ample canonical divisor is a log del Pezzo surface of Picard number one, while one with a numerically trivial canonical divisor is a log Enriques surface of Picard number one \cite{Keum-2018}. 

The following problem was raised by Koll\'ar:

\begin{problem}[{\cite[Problem 26]{Kollar-2008}}]\label{prob:Kollar}
    Classify all $\mathbb{Q}$-homology $\mathbb{CP}^2$'s with quotient singularities. 
\end{problem}

Although Koll\'ar expressed doubts about the feasiblity of this endeavor, there has been significant progress, particularly when the canonical divisor is not ample. In the case where the canonical divisor is anti-ample, the classification of log del Pezzo surfaces of Picard number 1 has been established \cite{Lacini-2024}. Some results are also available for the cases where the canonical divisor is numerically trivial \cite{Zhang-1991, Schutt-2023}. However, very little is known about situations where the canonical class is ample \cite{Hwang-Keum-2012}.

A $\mathbb{Q}$-homology $\mathbb{CP}^2$ whose smooth locus is simply-connected is of particular interest due to the following conjecture, called the \emph{algebraic Montgomery-Yang problem}. This conjecture represents a special case of Problem \ref{prob:Kollar} and serves as a key motivation for this paper. 

\begin{conjecture}[Algebraic Montgomery-Yang Problem, {\cite[Conjecture 30]{Kollar-2008}}]\label{conj:aMY} Let $S$ be a $\mathbb{Q}$-homology $\mathbb{CP}^2$ with quotient singularities. If its smooth locus $S^0:= S \setminus \textup{Sing}(S)$ is simply-connected, then $S$ has at most $3$ singularities.
\end{conjecture}

Conjecture \ref{conj:aMY} has been confirmed by Hwang and Keum, except when $S$ is a rational surface with an ample canonical divisor and only cyclic singularities \cite{Hwang-Keum-2011-1, Hwang-Keum-2013, Hwang-Keum-2014}. Recently, the authors showed that if $S$ is a $\mathbb{Q}$-homology $\mathbb{CP}^2$ whose smooth locus is simply-connected and contains four cyclic singularities, then the orders of the local fundamental groups at the singularities must be $2$, $3$, $5$, and $n$ with $n\geq 2599$ \cite{Jo-Park-Park-2024}.

It is natural to pose the following problem, which generalizes Conjecture \ref{conj:aMY}, and serves as a special case of Problem \ref{prob:Kollar}.

\begin{problem}\label{prob:simply_connected} Classify all $\mathbb{Q}$-homology $\mathbb{CP}^2$'s with quotient singularities whose smooth locus is simply-connected.     
\end{problem}

Alternatively, one may consider the following generalization of Problem \ref{prob:simply_connected}.

\begin{problem}\label{prob:trivial_H1} Classify all $\mathbb{Q}$-homology $\mathbb{CP}^2$'s with quotient singularities whose smooth locus has trivial first integral homology group.     
\end{problem}

It is worth noting that there exists a $\mathbb{Q}$-homology $\mathbb{CP}^2$ with quotient singularities whose smooth locus has trivial first integral homology group but nontrivial fundamental group \cite{Hwang-Keum-2011-1}. However, we are not aware of any such examples when all the singularities are cyclic.

For a normal projective surface $S$ with quotient singularities, the \textit{index} of a singularity $p\in \textrm{Sing}(S)$ is the smallest positive integer $k$ such that the divisor $kK_S$ is Cartier in a neighborhood of $p$. Note that $(S,p)$ is locally analytically isomorphic to $(\mathbb{C}^2/G_p,0)$ for some finite subgroup $G_p\subset \textrm{GL}(2,\mathbb{C})$, which acts freely on $\mathbb{C}^2\setminus\{0\}$. The index of $p$ is equal to the index $[G_p:G_p\cap \textrm{SL}(2,\mathbb{C})]$. The \textit{index} of $S$ is the smallest positive integer $k$ such that $kK_S$ is Cartier, and it is equal to the least common multiple of the indices of all singularities of $S$.

Rational homology projective planes with small indices have been studied by several authors. For example, log del Pezzo $\mathbb{Q}$-homology $\mathbb{CP}^2$'s of index one have been classified in \cite{Furushima-1986, Miyanishi-Zhang-1988, Ye-2002}, those of index two in \cite{Alexeev-Nikulin-1988, Alexeev-Nikulin-1989, Alexeev-Nikulin-2006, Kojima-2003}. Log Enriques $\mathbb{Q}$-homology $\mathbb{CP}^2$'s of index one have been classified in \cite{Hwang-Keum-Ohashi-2015, Schutt-2023}.

The aim of this article is to study $\mathbb{Q}$-homology $\mathbb{CP}^2$'s (without imposing any condition on the canonical divisor) of indices $\leq 3$. In particular, we provide a complete classification of the singularity types of $\mathbb{Q}$-homology $\mathbb{CP}^2$'s of indices one and two whose smooth locus has trivial first integral homology group or is
simply-connected. These classification results offer solutions to Problem \ref{prob:simply_connected} and Problem \ref{prob:trivial_H1} in terms of singularity types for these indices.

\subsection*{Singularity types of $\mathbb{Q}$-homology $\mathbb{CP}^2$'s of index one}

A $\mathbb{Q}$-homology $\mathbb{CP}^2$ of index one is called \textit{Gorenstein}. A quotient singularity is of index one if and only if it is a rational double point, i.e., a singularity of type  $A_n$, $D_n$, or $E_n$. Therefore, all singularities of a Gorenstein $\mathbb{Q}$-homology $\mathbb{CP}^2$ are rational double points. For example, we say that the singularity type of a $\mathbb{Q}$-homology $\mathbb{CP}^2$ is $2A_33A_1$ if it has two singularities of type $A_3$ and three singularities of type $A_1$. The singularity types of Gorenstein $\mathbb{Q}$-homology $\mathbb{CP}^2$'s are classified as follows.

\begin{theorem}[{\cite[Theorem 1.1]{Hwang-Keum-Ohashi-2015}}]\label{thm:Gorenstein_list} The singularity type of a Gorenstein $\mathbb{Q}$-homology $\mathbb{CP}^2$ is one of the following $58$ types: \begin{enumerate}[label=\textup{(\arabic*)}]
    \item $K\not\equiv 0$ $(27$ types$):$ $A_8$, $A_7A_1$, $A_5A_2A_1$, $2A_4$, $2A_32A_1$, $4A_2$, $E_8$, $E_7A_1$, $E_6A_2$, $D_8$, $D_62A_1$, $D_5A_3$, $2D_4$, $A_7$, $A_5A_2$, $2A_3A_1$, $E_7$, $D_6A_1$, $D_43A_1$, $A_5A_1$, $3A_2$, $E_6$, $A_32A_1$, $D_5$, $A_4$, $A_2A_1$, $A_1$.  
    \item  $K\equiv 0$ $(31$ types$):$ $A_9$, $A_8A_1$, $A_7A_2$, $A_72A_1$, $A_6A_2A_1$, $A_5A_4$, $A_5A_3A_1$, $A_52A_2$, $A_5A_22A_1$, $2A_4A_1$, $A_4A_32A_1$, $3A_3$, $2A_3A_2A_1$, $2A_33A_1$, $A_33A_2$, $D_9$, $D_8A_1$, $D_72A_1$, $D_6A_3$, $D_6A_2A_1$, $D_63A_1$, $D_5A_4$, $D_5A_3A_1$, $D_5D_4$, $D_4A_32A_1$, $2D_4A_1$, $E_8A_1$, $E_7A_2$, $E_72A_1$, $E_6A_3$, $E_6A_2A_1$.
\end{enumerate}
\end{theorem}

The list is divided into two non-overlapping sublists based on whether the canonical class $K$ is numerically trivial or not. It is known that each of these 58 types is realizable (see \cite{Miyanishi-Zhang-1988} for the case $K\not\equiv 0$, and \cite{Schutt-2023} for the case $K\equiv 0$). 

In the case where $-K$ is ample (the log del Pezzo case), the singularity types that can be realized by a Gorenstein $\mathbb{Q}$-homology $\mathbb{CP}^2$ with simply-connected smooth locus have been classified:

\begin{theorem}[{\cite{Miyanishi-Zhang-1988, Gurjar-Pradeep-Zhang-2002}}]\label{thm:del_Pezzo_Gorenstein} The singularity type of a log del Pezzo Gorenstein $\mathbb{Q}$-homology $\mathbb{CP}^2$ with simply-connected smooth locus is one of the following seven types, each of which is realizable: \[
E_8, ~E_7,~ E_6,~ D_5, ~A_4,~ A_2A_1, ~A_1.
\]
\end{theorem}

Our first result is to show that these are the only singularity types that can be realized by a (not necessarily log del Pezzo) Gorenstein $\mathbb{Q}$-homology $\mathbb{CP}^2$ whose smooth locus has trivial first integral homology group. 

\begin{theorem}\label{thm:Gorenstein_H1=0} The singularity type of a Gorenstein $\mathbb{Q}$-homology $\mathbb{CP}^2$ whose smooth locus has trivial first integral homology group is one of the following seven types, each of which is realizable: \[
E_8, ~E_7,~ E_6,~ D_5, ~A_4,~ A_2A_1, ~A_1.
\]
\end{theorem}

Furthermore, combining with Theorem \ref{thm:del_Pezzo_Gorenstein}, we obtain the following corollary. Therefore, Theorem \ref{thm:del_Pezzo_Gorenstein} holds even without the assumption that the canonical divisor is anti-ample.

\begin{corollary}\label{cor:Gorenstein_pi1=1} The singularity type of a Gorenstein $\mathbb{Q}$-homology $\mathbb{CP}^2$ with simply-connected smooth locus is one of the following seven types, each of which is realizable: \[
E_8, ~E_7,~ E_6,~ D_5, ~A_4,~ A_2A_1, ~A_1.
\]
\end{corollary}

\subsection*{Singularity types of $\mathbb{Q}$-homology $\mathbb{CP}^2$'s of index two}
A quotient singularity is said to be of type $K_n$ if the reduced exceptional divisor of its minimal resolution consists of rational curves and has the weighted dual graph depicted in Figure \ref{fig:weighted_dual_graph_of_K_n}. A quotient singularity is of index 2 if and only if it is of type $K_n$ \cite[Section 1.2]{Alexeev-Nikulin-2006}. 

\begin{figure}[!t]
\centering
\begin{tikzpicture}[scale=1]
    \draw (-9,0) node{$K_1$:};
    
    \draw (-8,0) node[circle, fill, inner sep=1.2pt, black]{};
    
    \draw (-8,0) node[below]{$-4$};
    
    \draw (-3.5,0) node{$K_n$ $(n\geq 2)$:};
    \draw (-2,0) node[circle, fill, inner sep=1.2pt, black]{};
    \draw (-1,0) node[circle, fill, inner sep=1.2pt, black]{};
    \draw (1,0) node[circle, fill, inner sep=1.2pt, black]{};
    \draw (2,0) node[circle, fill, inner sep=1.2pt, black]{};
    
    \draw (-2,0) node[below]{$-3$};
    \draw (-1,0) node[below]{$-2$};
    \draw (1,0) node[below]{$-2$};
    \draw (2,0) node[below]{$-3$};
    
    \draw (0,0) node{$\cdots$};
    
    \draw (-2,0)--(-1,0) (-1,0)--(-0.5,0) (0.5,0)--(1,0)  (1,0)--(2,0) ;
    \draw [thick,decorate,decoration={mirror,brace,amplitude=3pt},xshift=0pt]
    	(-1,-0.5) -- (1,-0.5) node [black,midway,xshift=0pt,yshift=-12pt] 
    	{$n-2$};
\end{tikzpicture}
\caption{The weighted dual graph of a singularity of type $K_n$.}
\label{fig:weighted_dual_graph_of_K_n}
\end{figure}
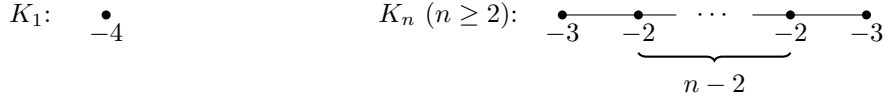

In contrast to the case of index 1, a complete classification of the singularity types that can occur on a $\mathbb{Q}$-homology $\mathbb{CP}^2$ of index 2 remains unknown. However, Alexeev and Nikulin provided a classification of log del Pezzo surfaces of index $2$ \cite{Alexeev-Nikulin-1988, Alexeev-Nikulin-1989, Alexeev-Nikulin-2006}. In particular, the types of singularities on log del Pezzo surfaces of index 2 and Picard number 1 are as follows.

\begin{theorem}[{\cite[Theorem 4.2]{Alexeev-Nikulin-2006}}]\label{thm:Alexeev-Nikulin} The singularity type of a log del Pezzo $\mathbb{Q}$-homology $\mathbb{CP}^2$ of index $2$ is one of the following $18$ types, each of which is realizable: 
$K_9$, $K_6A_2$, $K_5A_4$, $K_5$, $K_3A_5A_1$, $K_32A_2$, $K_2A_7$, $K_22A_3$, $K_2A_2$, $2K_1A_7$, $K_1D_8$, $K_1D_62A_1$, $K_1D_5A_3$, $K_12D_4$, $K_1A_7$, $K_1A_5A_1$, $K_1A_4$, $K_1$.   
\end{theorem}

Later, Kojima provided a classification of log del Pezzo surfaces of index 2 and Picard number 1, using a different approach from Alexeev and Nikulin. He also calculated the fundamental group of the smooth locus in each case. The following theorem is a part of \cite[Theorem 1.1]{Kojima-2003}.

\begin{theorem}[{\cite{Kojima-2003}}]\label{thm:Kojima}
The singularity type of a log del Pezzo $\mathbb{Q}$-homology $\mathbb{CP}^2$ of index $2$ whose smooth locus is simply-connected is one of the following four types, each of which is realizable: \[
K_5, ~K_2A_2,~ K_1A_4, ~K_1.
\]
\end{theorem}

Similar to the case of index $1$, our next result shows that these are the only singularity types that can be realized by a (not necessarily log del Pezzo) $\mathbb{Q}$-homology $\mathbb{CP}^2$ of index $2$, provided its smooth locus has trivial first integral homology group.

\begin{theorem}\label{thm:index2_list_H1=0} The singularity type of a  $\mathbb{Q}$-homology $\mathbb{CP}^2$ of index $2$ whose smooth locus has trivial first integral homology group is one of the following four types, each of which is realizable: \[
K_5, ~K_2A_2,~ K_1A_4, ~K_1.
\]
\end{theorem}

As a consequence, we obtain the following corollary which generalizes Theorem \ref{thm:Kojima}.

\begin{corollary}\label{cor:index2_list_pi1=1} The singularity type of a  $\mathbb{Q}$-homology $\mathbb{CP}^2$ of index $2$ whose smooth locus is simply-connected is one of the following four types, each of which is realizable: \[
K_5, ~K_2A_2,~ K_1A_4, ~K_1.
\]
\end{corollary}

\begin{remark}\label{rmk:thm_index2_H1=0} 
 The types $K_1A_4$ and $K_1$ cannot be realized by a $\mathbb{Q}$-homology $\mathbb{CP}^2$ whose canonical divisor is ample because these two types do not satisfy the orbifold BMY inequality (Theorem \ref{thm:oBMY}(1)). We do not know whether the types $K_5$ and $K_2A_2$ can be realized by a $\mathbb{Q}$-homology $\mathbb{CP}^2$ with ample canonical divisor.
\end{remark}

\subsection*{Singularity types of $\mathbb{Q}$-homology $\mathbb{CP}^2$'s of index three} We also provide a partial classification of the singularity types that can be realized by a $\mathbb{Q}$-homology $\mathbb{CP}^2$ of index $3$ whose smooth locus has trivial first integral homology group.

\begin{theorem}\label{thm:index3_list} The singularity type of a $\mathbb{Q}$-homology $\mathbb{CP}^2$ of index $3$ whose smooth locus has trivial first integral homology group is one of the following $18$ types. All types, except possibly $A_2(1,2)E_7$ and $A_2(2,2)E_8$, are realizable: \begin{align*}
    &A_1(1)E_8,~A_1(1)D_7,~ A_1(1)A_6,~A_1(1)A_4A_1,~ A_1(1)A_3,~A_1(1),~ A_6(1,1),~ A_{10}(1,1),~A_2(1,2)E_7, \\ &A_2(1,2)A_4,\ A_2(1,2)A_1,\ A_4(1,2)A_1,\ A_1(2)A_6,\ A_1(2),\ A_2(2,2)E_8,\ A_2(2,2)A_3,\ A_6(2,2),\  D_5(2).
\end{align*}
\end{theorem} 

The notation for the singularity types will be introduced in Section \ref{sec:index3_case}.
We do not know whether either of the two types, $A_2(1,2)E_7$ and $A_2(2,2)E_8$, is realizable. It will be shown in Section \ref{sec:index3_case} that each of the $16$ realizable cases can indeed be realized by a $\mathbb{Q}$-homology $\mathbb{CP}^2$ with simply-connected smooth locus.

\begin{corollary} The singularity type of a $\mathbb{Q}$-homology $\mathbb{CP}^2$ of index $3$ whose smooth locus is simply-connected is one of the following $18$ types. All types, except possibly $A_2(1,2)E_7$ and $A_2(2,2)E_8$, are realizable: \begin{align*}
    &A_1(1)E_8,~A_1(1)D_7,~ A_1(1)A_6,~A_1(1)A_4A_1,~ A_1(1)A_3,~A_1(1),~ A_6(1,1),~ A_{10}(1,1),~A_2(1,2)E_7, \\ &A_2(1,2)A_4,\ A_2(1,2)A_1,\ A_4(1,2)A_1,\ A_1(2)A_6,\ A_1(2),\ A_2(2,2)E_8,\ A_2(2,2)A_3,\ A_6(2,2),\  D_5(2).
\end{align*}
\end{corollary} 

\begin{remark}\label{rmk:index3_cor}
The canonical divisor of each of the $16$ realizations presented in Section \ref{sec:index3_case} is anti-ample. Among these, the $10$ types, $A_1(1)A_4A_1$, $A_1(1)A_3$, $A_1(1)$, $A_2(1,2)A_4$, $A_2(1,2)A_1$, $A_4(1,2)A_1$, $A_1(2)A_6$, $A_1(2)$, $A_2(2,2)A_3$, and $D_5(2)$, cannot be realized by a $\mathbb{Q}$-homology $\mathbb{CP}^2$ with an ample canonical divisor, as they do not satisfy the orbifold BMY inequality (Theorem \ref{thm:oBMY}(1)). Whether the remaining $8$ types, including $A_2(1,2)E_7$ and $A_2(2,2)E_8$, can be realized by a $\mathbb{Q}$-homology $\mathbb{CP}^2$ with an ample canonical divisor remains unknown.

On the other hand, Hwang \cite{Hwang-personal} has confirmed that there does not exist a log del Pezzo $\mathbb{Q}$-homology $\mathbb{CP}^2$ of index $3$ whose singularity type is $A_2(1,2)E_7$ or $A_2(2,2)E_8$. Therefore, if a $\mathbb{Q}$-homology $\mathbb{CP}^2$ has singularity type $A_2(1,2)E_7$ or $A_2(2,2)E_8$, its canonical divisor must be ample.
\end{remark}

The proof of our main classification results relies on the study of topological and smooth 4-manifolds, a central idea in \cite{Jo-Park-Park-2024}. We briefly outline our approach here: Fixing the index, we first generate a finite list of all possible singularity types by applying fundamental properties and known results of $\mathbb{Q}$-homology $\mathbb{CP}^2$'s. Next, we observe that, for a $\mathbb{Q}$-homology $\mathbb{CP}^2$, each singularity has a neighborhood homeomorphic to the cone on its link. By excising these cone neighborhoods from $S$, we obtain a compact, smooth (real) 4-dimensional manifold with $b_2=b_2^+=1$, whose boundary is the disjoint union of the links with opposite orientation. 

We then apply various tools from the theory of topological and smooth 4-manifolds such as linking form, Donaldson's diagonalization theorem, and Heegaard Floer $d$-invariants to demonstrate, through case-by-case analysis, the non-existence of such a 4-manifold in certain cases. Finally, we confirm that the remaining singularity types, which could not be eliminated, are realizable on a $\mathbb{Q}$-homology $\mathbb{CP}^2$, thereby completing the proof.

\subsection*{Acknowledgements} The authors express their gratitude to DongSeon Hwang for introducing this topic. They also thank all members of the 4-manifold topology group at Seoul National University (SNU) for their invaluable comments and insights throughout this work. Jongil Park was supported by the National Research Foundation of Korea (NRF) grant funded by the Korean government (No.2020R1A5A1016126 and RS-2024-00392067). 
He is also affiliated with the Research Institute of Mathematics at SNU. Kyungbae Park was supported by NRF grant funded by the Korean government (No.2021R1A4A3033098 and No.2022R1F1A1071673).

\section{Preliminaries}\label{sec:preliminaries}

In this section, we compile a variety of techniques that can be used to reduce the list of possible singularity types on rational homology projective planes. The properties applied here are drawn from multiple fields; in particular, we address the orbifold BMY inequality from algebraic geometry and Donaldson's diagonalization theorem from the theory of smooth 4-manifolds.

\subsection{Results from Algebraic Geometry}
For a normal projective surface $S$ with quotient singularities, its \emph{orbifold Euler characteristic}, denoted by $e_{\textup{orb}}(S)$, is defined by the following formula: 
\[ 
    e_{\textup{orb}}(S):= e(S) - \sum_{p\in \textup{Sing}(S)} \left(1-\frac{1}{|G_p|} \right),
\]
where $e(S)$ is the topological Euler characteristic of $S$, and $G_p$ is the local fundamental group at the point $p$, which is the fundamental group of the link $L_p$ of $S$ at $p$.

\begin{theorem}\label{thm:oBMY} 
    Let $S$ be a normal projective surface with quotient singularities. 
    \begin{enumerate}[label=\normalfont(\arabic*)]
        \item \textup{(Orbifold BMY inequality) \cite{KoNS-1989,Megyesi-1999,Miyaoka-1984,Sakai-1980}}  If the canonical class $K_S$ of $S$ is nef, then 
        \begin{equation*} 
            K^2_S \leq 3e_{\textup{orb}}(S).
        \end{equation*}
        In particular, \[0\leq e_{\textup{orb}}(S).\]
        \item \textup{(Weak orbifold BMY inequality) \cite{KeM-1999}}  If $-K_S$ is nef, then 
        \[ 0\leq e_{\textup{orb}}(S). \]
    \end{enumerate} 
\end{theorem}

Recall that if $S$ is a $\mathbb{Q}$-homology $\mathbb{CP}^2$, then either $K_S$ is nef or $-K_S$ is ample. Thus, by Theorem \ref{thm:oBMY}, we have $e_{\textup{orb}}(S)\geq 0$. Since $e(S)=e(\mathbb{CP}^2)=3$ and $|G_p|\geq 2$ for each $p\in \textup{Sing}(S)$, it follows directly that $|\textup{Sing}(S)|\leq 6$. In fact, the following result is known, which combines \cite[Theorem 1.2]{Belousov-2008} and \cite[Theorem 1.1]{Hwang-Keum-2011-2} (see also \cite[Theorem 2.2]{Hwang-Keum-Ohashi-2015}).

\begin{theorem}\label{thm:5_quot_sing} Let $S$ be a $\mathbb{Q}$-homology $\mathbb{CP}^2$ with quotient singularities. Then $S$ has at most five singularities. Moreover, if $S$ has exactly five singularities, it must be Gorenstein with singularity type $2A_33A_1$.   
\end{theorem}

Let $S$ be a $\mathbb{Q}$-homology $\mathbb{CP}^2$, and let $f\colon\tilde{S}\to S$ be the minimal resolution. Denote the smooth locus of $S$ by $S^0:=S\setminus \textrm{Sing}(S)$, and assume that $H_1(S^0;\mathbb{Z})=0$. Then by \cite[Lemma 3]{Hwang-Keum-2011-1}, the group $H_2(\tilde{S};\mathbb{Z})$ is free, and the intersection lattice on $H_2(\tilde{S};\mathbb{Z})$ is unimodular since $\tilde{S}$ is a closed, real 4-dimensional manifold. 

For a singularity $p\in \textup{Sing}(S)$, let $R_p$ denote the negative definite sublattice generated by the numerical classes of the components of $f^{-1}(p)$. The absolute value $|\det(R_p)|$ of its determinant equals the order $|G_p/[G_p,G_p]|=|H_1(L_p;\mathbb{Z})|$, where $L_p$ is the link of $S$ at $p$ and $G_p$ is the local fundamental group of $S$ at $p$, i.e., the fundamental group of $L_p$. For quotient singularities of index $\leq 2$, the values of $|G_p|$ and $|\det(R_p)|$ are given in Table \ref{tab:|det|}.

\begin{table}[t]
\centering
\begin{tabular}{c|c|c}
      Type  & $|G_p|$ &   $|\det(R_p)|$  \\ [-1em] &&\\ 
      \hline &&\\ [-1em]  
       $A_n$ $(n\geq 1)$ & $n+1$ & $n+1$  \\ [-1em] &&\\
       $D_n$ $(n\geq 4)$ & $4(n-2)$ & $4$  \\ [-1em] &&\\
       $E_6$ & $24$ & $3$  \\ [-1em]&& \\
       $E_7$ & $48$ & $2$ \\ [-1em]&& \\
       $E_8$ & $120$ & 1  \\ [-1em]&&\\
       $K_n$ $(n\geq 1)$ & $4n$ & $4n$ 
\end{tabular}
    \caption{$|G_p|$ and $|\det(R_p)|$ for singularities of types $A_n$, $D_n$, $E_n$, and $K_n$.}
    \label{tab:|det|} 
\end{table}

\begin{lemma}\label{lem:canonical_divisor} Let $S$ be a $\mathbb{Q}$-homology $\mathbb{CP}^2$ with smooth locus $S^0$.
\begin{enumerate}[label=\textup{(\arabic*)}]
    \item \cite[Lemma 3]{Hwang-Keum-2011-1} If $H_1(S^0;\mathbb{Z})=0$, then the numbers $|\det(R_p)|$ are pairwise relatively prime and the canonical divisor $K_S$ is not numerically trivial, i.e., $K_S$ is either ample or anti-ample. 
    \item \cite[Lemma 3.3]{Hwang-Keum-2011-2}, \cite[Lemma 2.3]{Hwang-Keum-Ohashi-2015} If $K_S$ is not numerically trivial, then \[
    D:=K_S^2 \cdot \prod_{p\in \textup{Sing}(S)}|\det(R_p)|
    \] 
    is a nonzero square number.   
\end{enumerate}
\end{lemma}

\subsection{Topological and Smooth Obstructions}\label{subsec:top_and_smooth_obstructions}
Let $S$ be a $\mathbb{Q}$-homology $\mathbb{CP}^2$, and let $f\colon\tilde{S}\to S$ be the minimal resolution. Each singularity $p\in \textrm{Sing}(S)$ has a neighborhood which is homeomorphic to the cone on $L_p$, where $L_p$ is the link of $S$ at $p$. Note that each $L_p$ is a rational homology $3$-sphere. By removing these cone neighborhoods, we obtain a compact, smooth $4$-manifold with $b_2=1$, which we denote by $W^0$. 

Observe that $W^0$ is homotopy equivalent to $S^0$, as $S^0$ deformation retracts onto $W^0$. The manifold $W^0$ naturally embeds in $\tilde{S}$; in fact $W^0$ is the complement of a regular neighborhood of $\bigcup_{p\in \textrm{Sing}(S)} f^{-1}(p)$ in $\tilde{S}$. With the orientation induced from the canonical orientation of $\tilde{S}$, the intersection form of $W^0$ is positive definite, i.e., $b_2^+(W^0)=b_2(W^0)=1$ \cite[Theorem 1.4.13]{Gompf-Stipsicz-1999}. The boundary $\partial W^0$ is the disjoint union $\coprod_{p\in \textrm{Sing}(S)}\left(-L_p\right)$, consisting of the links of singularities with opposite orientation.

In the case where $|\textrm{Sing}(S)|\geq 2$, by removing thickened arcs (i.e., 4-dimensional 3-handles) from $W^0$, we obtain a compact, oriented, smooth 4-manifold $Z^0$ with $b_2^+(Z^0)=b_2(Z^0)=1$ and $H_1(Z^0;\mathbb{Z})\cong H_1(W^0;\mathbb{Z})$. The intersection form of $Z^0$ is isomorphic to that of $W^0$, and the boundary of $Z^0$ is the connected sum $\#_{p\in \textrm{Sing}(S)} \left(-L_p\right)$.

\subsubsection{Topological Constraints}
A fundamental topological obstruction arises from the relationship between the intersection form of a $4$-manifold and the linking form of its boundary $3$-manifold.
\begin{definition}[Linking form] Let $Y$ be an oriented rational homology $3$-sphere. For disjoint oriented two knots $K_1, K_2$ in $Y$, choose a rational $2$-chain $c$ such that $\partial c=K_1$, and define \[ 
    \lambda_Y([K_1],[K_2])= c \cdot K_2. \]
This yields a well-defined nondegenerate symmetric bilinear form \[
    \lambda_Y\colon H_1(Y;\mathbb{Z})\times H_1(Y;\mathbb{Z})\to \mathbb{Q}/\mathbb{Z}.\] 
See \cite[Exercise 4.5.12(c)]{Gompf-Stipsicz-1999} for example. 
\end{definition}

The following lemma is well known; see \cite[Exercise 5.3.13(f),(g)]{Gompf-Stipsicz-1999} for example.

\begin{lemma}\label{lem:linking_form} Let $Y$ be a rational homology $3$-sphere. \begin{enumerate}[label=\textup{(\arabic*)}]
    \item Suppose $X$ is a compact, oriented $4$-manifold with $H_1(X;\Bbb Z)=0$ and $\partial X=Y$. If $A$ is a matrix representing the intersection form of $X$, then $H_1(Y;\Bbb Z)$ is isomorphic to the cokernel of $A\colon\Bbb Z^{b_2(X)}\to \Bbb Z^{b_2(X)}$, and $(-A)^{-1}$ represents the linking form on $H_1(Y;\Bbb Z)$. 
    \item If $Y$ is obtained from $S^3$ by integral surgery along an oriented, framed link $L\subset S^3$ with linking matrix $B$, then $(-B)^{-1}$ represents the linking form on $H_1(Y;\Bbb Z)$, with respect to the generating set given by meridians of the link components.
\end{enumerate}
\end{lemma}

For a compact, oriented $4$-manifold $X$ and its intersection form \[
    Q_X\colon H_2(X;\Bbb Z)/\textrm{Tor}\times H_2(X;\Bbb Z)/\textrm{Tor}\to \Bbb Z,\] 
we simply denote the lattice $Q_X:=(H_2(X;\Bbb Z)/\textrm{Tor}, Q_X)$.

\begin{corollary}\label{cor:linking_form} Let $S$ be a $\mathbb{Q}$-homology $\mathbb{CP}^2$ such that $H_1(S^0;\mathbb{Z})=0$. Then the intersection forms $Q_{W^0}$ and $Q_{Z^0}$ of $W^0$ and $Z^0$, respectively, are both represented by the $1\times 1$ matrix $(n)$, where \[
n=\prod_{p\in \textup{Sing}(S)} |H_1(L_p;\mathbb{Z})|.
\]
Moreover, $H_1(\partial W^0;\mathbb{Z})\cong H_1(\partial Z^0;\mathbb{Z})$ is a cyclic group of order $n$, and $H_1(L_p;\mathbb{Z})$ is cyclic for each $p\in \textup{Sing}(S)$. 
\end{corollary}
\begin{proof} Recall that $Q_{Z^0}\cong Q_{W^0}$. Since $b_2^+(Z^0)=b_2(Z^0)=1$, the intersection form $Q_{Z^0}$ is represented by a $1\times 1$ matrix $(n)$ for some positive integer $n$. By Lemma \ref{lem:linking_form}(1), $H_1(\partial Z^0;\mathbb{Z})$ is isomorphic to $\mathbb{Z}_n$. Since $\partial Z^0=\#_{p\in\textrm{Sing}(S)} (-L_p)$, it follows that each $H_1(L_p;\mathbb{Z})$ is cyclic, and that \[
n=|H_1\left(\partial Z^0;\mathbb{Z}\right)|=\left|\bigoplus_{p\in \textup{Sing}(S)} H_1(L_p;\mathbb{Z}) \right|=\prod_{p\in \textup{Sing}(S)} |H_1(L_p;\mathbb{Z})|.
\]
\end{proof}

\subsubsection{Donaldson's Diagonalization Theorem}\label{subsubsec:Donaldson}
Donaldson's diagonalization theorem imposes a significant constraint on the intersection forms of closed, oriented, \textit{smooth}, \textit{definite} $4$-manifolds. This constraint can be utilized to derive a condition for a smooth $4$-manifold with a specified $3$-manifold as its boundary, particularly when it bounds a smooth definite $4$-manifold. The argument presented here is inspired by Lisca's remarkable work \cite{Lisca-2007}.

For a positive integer $n$, let $\{e_1,\dots,e_n\}$ be the standard basis for $\Bbb Z^n$. We denote by $-\Bbb Z^n$ the standard negative definite lattice $(\Bbb Z^n, \langle\cdot,\cdot\rangle)$, where the pairing is given by $\langle e_i,e_j\rangle = -\delta_{i,j}$, with $\delta_{i,j}$ being the Kronecker delta. 

\begin{theorem}[Donaldson's Diagonalization Theorem, {\cite{Donaldson-1983,Donaldson-1987}}]\label{thm:Donaldson} If the intersection form $Q_W$ of a closed, oriented, smooth $4$-manifold $W$ is negative definite, then $Q_W$ is isomorphic to $-\Bbb Z^n$, where $n=b_2(W)=b_2^-(W)$.     
\end{theorem}

Then we obtain the following corollary for $\mathbb{Q}$-homology $\mathbb{CP}^2$ with quotient singularities.

\begin{corollary}\label{cor:Donaldson} Let $S$ be a $\mathbb{Q}$-homology $\mathbb{CP}^2$. Suppose that for each $p\in \textup{Sing}(S)$, the orientation reversal $-L_p$ of the link $L_p$ of $S$ at $p$ bounds a compact, oriented, negative definite, smooth $4$-manifold $X_p$. Then the lattice $\bigoplus_{p\in \textup{Sing}(S)} Q_{X_p}$ embeds into $-\mathbb{Z}^{n+1}$, where $n=\sum_{p\in \textup{Sing}(S)} b_2(X_p)$. 
\end{corollary}

\begin{proof} Consider the smooth 4-manifold $W^0$ constructed as above. The disjoint union $X:=\coprod_{p\in \textrm{Sing}(S)} X_p$ is a negative definite 4-manifold whose boundary is $\coprod_{p\in \textrm{Sing}(S)} \left(-L_p\right)$ and satisfies $Q_X \cong \bigoplus_{p\in \textup{Sing}(S)} Q_{X_p}$. Now define $W:=(-W^0)\cup_{\partial} X$. This is a closed, oriented, smooth, negative definite 4-manifold with $b_2(W)=n+1$. By Theorem \ref{thm:Donaldson}, $Q_W$ is isomorphic to $-\mathbb{Z}^{n+1}$. Since $H_2(L_p;\mathbb{Z})=0$ for each $p\in \textup{Sing}(S)$, the inclusions of $-W^0$ and $X$ into $W$ induce an embedding $\iota\colon Q_{-W^0}\oplus Q_X \hookrightarrow Q_W$ of lattices. \end{proof}

In the case where the smooth locus has trivial first integral homology group, we can provide an additional condition using the following proposition.

\begin{proposition}[{\cite[Lemma 2.4]{AMP-2022}}]\label{prop:orthogonal_complement} Let $Y$ be an oriented $3$-manifold with $H^1(Y;\mathbb{Z})=0$ which is the boundary of two compact, oriented $4$-manifolds $X_1$ and $X_2$ with $H_1(X_1;\Bbb Z)=0$. If $X$ is a closed, oriented $4$-manifold obtained by $X:=X_1\cup_Y (-X_2)$, then the inclusions $X_1,-X_2\hookrightarrow X$ induce an embedding of lattices \[
  \iota\colon Q_{X_1}\oplus (-Q_{X_2}) \to Q_X
\]
such that $\iota(-Q_{X_2})$ is the orthogonal complement of $\iota(Q_{X_1})$ in $Q_X$.     
\end{proposition}

Then, we have the following corollary for a $\mathbb{Q}$-homology $\mathbb{CP}^2$ with $H_1(S^0;\mathbb{Z})=0$.

\begin{corollary}\label{cor:orthogonal_complement} Let $S$ be a $\mathbb{Q}$-homology $\mathbb{CP}^2$ such that $H_1(S^0;\mathbb{Z})=0$. Suppose that, for each $p\in \textup{Sing}(S)$, the orientation reversal $-L_p$ of the link $L_p$ of $S$ at $p$ bounds a compact, oriented, negative definite, smooth $4$-manifold $X_p$ with $H_1(X_p;\mathbb{Z})=0$. Then there exists an embedding \[
    \iota\colon\bigoplus_{p\in \textup{Sing}(S)}Q_{X_p}\hookrightarrow -\mathbb{Z}^{n+1}\] 
of lattices, where $n=\sum_{p\in \textup{Sing}(S)} b_2(X_p)$, such that the generator (uniquely determined up to sign) of the orthogonal complement of $\iota\left(\bigoplus_{p\in \textup{Sing}(S)}Q_{X_p}\right)$ in $-\mathbb{Z}^{n+1}$ has square $-\prod_{p\in \textup{Sing}(S)}|H_1(L_p;\mathbb{Z})|$.   
\end{corollary}
\begin{proof} Consider the situation in the proof of Corollary \ref{cor:Donaldson}. In this case $H_1(X;\mathbb{Z})=0$, so by Proposition \ref{prop:orthogonal_complement} $\iota(Q_{-W^0})$ is the orthogonal complement of $\iota(Q_X)$ in $Q_W\cong -\mathbb{Z}^{n+1}$. On the other hand, since $H_1(W^0;\mathbb{Z})=H_1(S^0;\mathbb{Z})=0$, we have $Q_{-W^0}\cong -Q_{W^0}$, which is represented by $\left(-\prod_{p\in \textup{Sing}(S)}|H_1(L_p;\mathbb{Z})|\right)$ by Corollary \ref{cor:linking_form}. The result now follows immediately.    
\end{proof}

For relatively prime integers $p>q>0$, the lens space $L(p,q)$ is an oriented 3-manifold obtained by $-p/q$-surgery along the unknot in $S^3$.  The fraction $p/q$ can be uniquely expressed as a Hirzebruch-Jung continued fraction:
\[ 
    \frac{p}{q}=[a_1,\dots,a_\ell]:=a_1-\frac{1}{a_2-\displaystyle\frac{1}{\cdots-\displaystyle\frac{1}{a_\ell}}} \ ~~(a_i\geq 2).
\]
It is well known that $L(p,q)$ is the boundary of the negative definite plumbed 4-manifold $X(p,q)$, constructed from the linear graph shown in Figure \ref{fig:plumbing_graph_of_X(p,q)} (see \cite[Exercise 5.3.9(b)]{Gompf-Stipsicz-1999}). Since $L(p,p-q)$ is homeomorphic to the orientation reversal $-L(p,q)$ of $L(p,q)$, we conclude that $-L(p,q)$ is the boundary of the negative definite 4-manifold $X(p,p-q)$. 

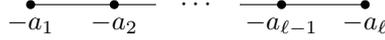
\begin{figure}[!t]
\centering
\begin{tikzpicture}[scale=1.1]
\draw (-2,0) node[circle, fill, inner sep=1.2pt, black]{};
\draw (-1,0) node[circle, fill, inner sep=1.2pt, black]{};
\draw (1,0) node[circle, fill, inner sep=1.2pt, black]{};
\draw (2,0) node[circle, fill, inner sep=1.2pt, black]{};

\draw (-2,0) node[below]{$-a_1$};
\draw (-1,0) node[below]{$-a_2$};
\draw (1,0) node[below]{$-a_{\ell-1}$};
\draw (2,0) node[below]{$-a_\ell$};

\draw (0,0) node{$\cdots$};

\draw (-2,0)--(-1,0) (-1,0)--(-0.5,0) (0.5,0)--(1,0)  (1,0)--(2,0) ;
\end{tikzpicture}
\caption{The plumbing graph for the negative definite 4-manifold $X(p,q)$ bounded by the lens space $L(p,q)$.}
\label{fig:plumbing_graph_of_X(p,q)}
\end{figure}

\begin{example}\label{ex:link_of_An_Kn}
    The link of a singularity of type $A_n$ is $L(n+1,n)$. Its orientation reversal $L(n+1,1)$ is the boundary of the 4-manifold $X(n+1,1)$, whose plumbing graph is shown in Figure \ref{fig:plumbing_graph_of_X(n+1,1)}.
    
\begin{figure}[!t]
\centering
\begin{tikzpicture}[scale=1]
    \draw (-8,0) node[circle, fill, inner sep=1.2pt, black]{};
    \draw (-8,0) node[below]{$-(n+1)$};
\end{tikzpicture}
\caption{The plumbing graph of $X(n+1,1)$.}
\label{fig:plumbing_graph_of_X(n+1,1)}
\end{figure}
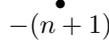

The link of a singularity of type $K_n$ is $L(4n,2n-1)$. Its orientation reversal $L(4n,2n+1)$ is the boundary of the 4-manifold $X(4n,2n+1)$, whose plumbing graph is shown in Figure \ref{fig:plumbing_graph_of_X(4n,2n+1)}.
    
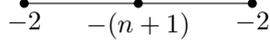
\begin{figure}[!th]
\centering
\begin{tikzpicture}[scale=1]
    \draw (-1.5,0) node[circle, fill, inner sep=1.2pt, black]{};
    \draw (0,0) node[circle, fill, inner sep=1.2pt, black]{};
    \draw (1.5,0) node[circle, fill, inner sep=1.2pt, black]{};
    
    \draw (-1.5,0) node[below]{$-2$};
    \draw (0,0) node[below]{$-(n+1)$};
    \draw (1.5,0) node[below]{$-2$};
    \draw (-1.5,0)--(1.5,0);
\end{tikzpicture}
\caption{The plumbing graph of $X(4n,2n+1)$.}
\label{fig:plumbing_graph_of_X(4n,2n+1)}
\end{figure}
\end{example}

\subsubsection{Heegaard Floer $d$-invariants}
A rational homology $3$-sphere $Y$ is defined as an \textit{L-space} if the rank of its hat version of Heegaard Floer homology, denoted by $\textrm{rank}(\widehat{HF}(Y))$, equals $|H_1(Y;\Bbb Z)|$. Lens spaces are known to be $L$-spaces, and the property of being an $L$-space is preserved under connected sums \cite{Ozsvath-Szabo-2005}. Moreover, it is known that the boundary of a negative definite 4-dimensional plumbing is an $L$-space if it is the link of a rational surface singularity \cite{Nemethi-2005}, \cite[p.358]{Lisca-Stipsicz-2007}.

A particularly relevant invariant in Heegaard Floer homology, which encapsulates information about $4$-dimensional manifolds bounded by $Y$, is the \textit{d-invariant}, also known as the \textit{correction term}. This is analogous to Fr\o yshov's invariant in Seiberg-Witten theory \cite{Froyshov-1996}. This rational-valued invariant is associated with a pair $(Y,\mathfrak{t})$, where $Y$ is a rational homology $3$-sphere, and $\mathfrak{t}$ is a spin$^c$ structure on $Y$. It is also known that $d$-invariant is additive under the connected sum of two spin$^c$ rational homology $3$-spheres. For further details, see \cite{Ozsvath-Szabo-2003}. 

In particular, for $d$-invariants of the boundary $3$-manifolds of a spin cobordism, especially those with small $b_2$, we have the following.

\begin{proposition}[{\cite[Lemma 2.7]{Lidman-Moore-Vazquez-2019}}]\label{prop:spin_cobordism} Let $(W,\mathfrak{s})\colon(Y,\mathfrak{t})\to (Y',\mathfrak{t}')$ be a spin cobordism between $L$-spaces satisfying $b_2^+(W)=1$ and $b_2^-(W)=0$. Then \[d(Y',\mathfrak{t}')-d(Y,\mathfrak{t})=-\frac{1}{4}.  \]    
\end{proposition}

The following is an immediate corollary of the fact that $d(S^3,\mathfrak{t})=0$, where $\mathfrak{t}$ is the unique spin$^c$ structure on $S^3$.

\begin{corollary}\label{cor:spin_d} Let $(W,\mathfrak{s})$ be a compact, oriented, smooth, spin $4$-manifold whose boundary $Y$ is an $L$-space. If $b_2^+(W)=1$ and $b_2^-(W)=0$, then 
\[d(Y,\mathfrak{s}|_Y)=-\frac{1}{4}.\]  
\end{corollary}

Applying this to a $\mathbb{Q}$-homology $\mathbb{CP}^2$, we obtain:
\begin{corollary}\label{cor:spin_d_2} Let $S$ be a $\mathbb{Q}$-homology $\mathbb{CP}^2$ such that $H_1(S^0;\mathbb{Z})=0$. If $\prod_{p\in \textup{Sing}(S)}|H_1(L_p;\mathbb{Z})|$ is even, then for each $p\in \textup{Sing}(S)$ there exists a spin structure $\mathfrak{s}_p$ on the link $L_p$ of $S$ at $p$ such that \[
\sum_{p\in \textup{Sing}(S)} d(L_p,\mathfrak{s}_p)=\frac{1}{4}.
\]
\end{corollary} 
\begin{proof} Consider the smooth 4-manifold $Z^0$ constructed in Section \ref{subsec:top_and_smooth_obstructions}, whose boundary is the connected sum $\#_{p\in \textup{Sing}(S)}\left(-L_p\right)$. Since $H_1(S^0;\mathbb{Z})=0$, we have $H_1(Z^0;\mathbb{Z})=0$, and the intersection form $Q_{Z^0}$ of $Z^0$ is represented by $\left(\prod_{p\in \textup{Sing}(S)}|H_1(L_p;\mathbb{Z})|\right)$ by Corollary \ref{cor:linking_form}. 

Since $\prod_{p\in \textup{Sing}(S)}|H_1(L_p;\mathbb{Z})|$ is even by assumption, it follows that $Z^0$ is a spin 4-manifold \cite[Corollary 5.7.6]{Gompf-Stipsicz-1999}. Thus, by Corollary \ref{cor:spin_d}, we have \[
d(\partial Z^0,\mathfrak{s})=-\frac{1}{4}\] 
for some spin structure $\mathfrak{s}$ of $\partial Z^0$. Note that every spin structure on $\partial Z^0=\#_{p\in \textup{Sing}(S)}(-L_p)$ is induced from a spin structure on each $-L_p$. Therefore, we conclude that there exists a spin structure $\mathfrak{s}_p$ on $L_p$ for each $p\in \textup{Sing}(S)$ such that \[
-\frac{1}{4}=d\biggl( \underset{p\in \textup{Sing}(S)}{\#}-L_p, \underset{p\in \textup{Sing}(S)}{\#}\mathfrak{s}_p\biggr)=\sum_{p\in\textup{Sing}(S)} d(-L_p,\mathfrak{s}_p)=\sum_{p\in \textup{Sing}(S)} -d(L_p,\mathfrak{s}_p).
\]
\end{proof}

Recall that our orientation of the lens space $L(p,q)$ is defined such that it is obtained by $-p/q$-surgery along the unknot in $S^3$. Following \cite[Proposition 4.8]{Ozsvath-Szabo-2003} (noting that their orientation convention differs from ours), we can choose an identification $\text{Spin}^c(L(p,q))\cong \Bbb Z_p$ such that the following recursive formula for the $d$-invariants applies: 
\begin{equation}\label{eq:d-invariant}
    d(L(p,q),i)=\frac{1}{4}-\frac{(2i+1-p-q)^2}{4pq}-d(L(q,r),j),
\end{equation}
assuming $0<q<p$ and $0\leq i<p$, where $r$ and $j$ are the reductions of $p$ and $i$ modulo $q$, respectively. The spin structures of $L(p,q)$ correspond exactly to the integers among $(q-1)/2$ and $(p+q-1)/2$ \cite[p.134]{Ue-2009}.

\begin{example}[Spin $d$-invariants of the link of a singularity of type $A_n$ and $K_n$]\label{ex:d_invs_An_Kn} Recall from Example \ref{ex:link_of_An_Kn} that the link of a singularity of type $A_n$ is $L(n+1,n)$. For odd $n$, $L(n+1,n)$ has two spin structures $(n-1)/2$ and $n$. The recursive formula (\ref{eq:d-invariant}) gives \[
d\left(L(n+1,n),\frac{n-1}{2}\right)=-\frac{1}{4}\quad \text{and}\quad d\left(L(n+1,n),n\right)=\frac{n}{4}.
\]
For even $n$, $L(n+1,n)$ has a unique spin structure $n$, and we have \[
d(L(n+1,n),n)=\frac{n}{4}.
\] 

The link of a singularity of type $K_n$ is $L(4n, 2n-1)$, which has two spin structures, $n-1$ and $3n-1$. We have 
\[
d(L(4n,2n-1),n-1)=\begin{cases}
    -3/4 & \text{for odd $n$}, \\ 
    -1/4 & \textrm{for even $n$},
\end{cases}\]
and 
\[ 
\quad d(L(4n,2n-1),3n-1)=\begin{cases}
    1/4 & \text{for odd $n$}, \\ 
    -1/4 & \text{for even $n$}.
\end{cases}
\]

 \end{example}

\begin{example}[$d$-invariants of the link of a singularity of type $D_n$]\label{ex:d_invs_Dn} Consider the link $L(D_n)$ of a singularity of type $D_n$. We have (see \cite[Satz 2.11]{Brieskorn-1968}): \[
H_1(L(D_n);\mathbb{Z})=\begin{cases}
    \mathbb{Z}_2\oplus \mathbb{Z}_2 & \textrm{if $n$ is even}, \\
    \mathbb{Z}_4 & \textrm{if $n$ is odd}.
\end{cases}
\]
In particular, $L(D_n)$ has four distinct spin$^c$ structures for each $n$. The corresponding $d$-invariants are given by $\displaystyle\frac{n}{4}$, $0$, $0$, and $\displaystyle\frac{n-4}{4}$ \cite[Example 15]{Doig-2015}. Notably, the values $\displaystyle\frac{n}{4}$ and $\displaystyle\frac{n-4}{4}$ correspond to spin structures \cite[Table 3]{Doig-2015}.
\end{example}

To determine $d$-invariants for the link of a singularity of type $E_n$ for $n=6,7,8$, we recall the work of Ni and Wu on Heegaard Floer homology of 3-manifolds obtained by surgery along a knot. Let $p,q>0$ be relatively prime integers and $i\in \{0,\dots,p-1\}$. Given a knot $K\subset S^3$, let $S^3_{p/q}(K)$ denote the 3-manifold obtained by $p/q$-surgery along $K$. A non-decreasing sequence $\{V_s(K)\}_{s=0}^\infty$ of non-negative integers can be assigned to $K$ \cite[Section 2]{Ni-Wu-2015}. It is known that \[
V_s(K)=\begin{cases}
    1 & \textrm{if $s=0$} \\
    0 & \textrm{if $s>0$}
\end{cases}
\]
when $K$ is the right-handed trefoil knot  \cite{Ozsvath-Szabo-2005-2}, \cite[p.238]{Choe-Park-2021}.
\begin{proposition}[{\cite[Proposition 1.6, Remark 2.10]{Ni-Wu-2015}}]\label{prop:Ni-Wu} The following formula holds: \[
d(S^3_{p/q}(K),i)=d(-L(p,q),i)-2\max\left\{V_{\lfloor \frac{i}{q}\rfloor}(K), V_{\lfloor\frac{p+q+1-i}{q}\rfloor}(K)\right\},
\]
where $x\mapsto \lfloor x\rfloor$ is the usual floor function.     
\end{proposition}

\begin{example}[$d$-invariants of the link of a singularity of type $E_n$]\label{ex:d_invs_En}

For $n=6,7,8$, it can be shown using elementary Kirby calculus that the link $L(E_n)$ of a singularity of type $E_n$ can be obtained from $S^3$ by $(n-9)$-surgery along the left-handed trefoil knot (see \cite[Exercise 5.1.12(a)]{Gompf-Stipsicz-1999}). (Note that the same argument shows that the link $L(D_5)$ of a singularity of type $D_5$ can be obtained from $S^3$ by $(-4)$-surgery along the left-handed trefoil knot.) In particular, $H^2(L(E_n);\mathbb{Z})=H_1(L(E_n);\mathbb{Z})\cong \mathbb{Z}_{9-n}$ $(n=6,7,8)$ by Lemma \ref{lem:linking_form}(2), and the $d$-invariants of $L(E_n)$ can be computed using Proposition \ref{prop:Ni-Wu}. The results are summarized in Table \ref{tab:d_invs}. We also note that $L(E_8)$ is the Poincar\'e homology sphere $\Sigma(2,3,5).$
\end{example}

\begin{table}[t]
    \centering
    \begin{tabular} {c|c}
Link  &  $d$-invariants \\ [-1em] \\ 
      \hline \\ [-0.8em] 
       $L(E_6)$ & $\dfrac{3}{2}$, $\dfrac{1}{6}$, $\dfrac{1}{6}$\\ [-0.5em]  \\
       $L(E_7)$ & $\dfrac{7}{4}$, $\dfrac{1}{4}$ \\ [-0.6em]\\
       $L(E_8)$ & $2$ 
       \end{tabular}
    \caption{$d$-invariants of the links of singularities of type $E_n$.}
    \label{tab:d_invs}
\end{table}

\section{The Index One Case (Proof of Theorem \ref{thm:Gorenstein_H1=0})}
Let $S$ be a Gorenstein $\mathbb{Q}$-homology $\mathbb{CP}^2$ such that $H_1(S^0;\mathbb{Z})=0$. By Lemma \ref{lem:canonical_divisor}(1), $K_S$ is not numerically trivial, so the singularity type of $S$ must be one of the $27$ types listed in Theorem \ref{thm:Gorenstein_list}(1) for the case $K\not\equiv 0$. Moreover, by Lemma \ref{lem:canonical_divisor}(1), the numbers $|\det(R_p)|$ are pairwise relatively prime (see Table \ref{tab:|det|}). It follows that the singularity type of $S$ must be one of the following $10$ types: \[
    A_8, ~E_8,~ D_8, ~A_7, ~E_7,~ E_6,~ D_5, ~A_4,~ A_2A_1, ~A_1.
\]

\begin{lemma} The type $D_8$ does not occur as the singularity type of $S$.
\end{lemma}
\begin{proof} Recall that $H_1(L(D_8);\mathbb{Z})=\mathbb{Z}_2\oplus \mathbb{Z}_2$ (Example \ref{ex:d_invs_Dn}), which is not cyclic. Therefore the type $D_8$ does not occur by Corollary \ref{cor:linking_form}.
\end{proof}

\begin{lemma} The types $A_8$ and $A_7$ do not occur as the singularity type of $S$.
\end{lemma}
\begin{proof} We apply Corollary \ref{cor:orthogonal_complement} in both cases.

(1) Type $A_8$: The link of a singularity of type $A_8$ is the lens space $L(9,8)$ (Example \ref{ex:link_of_An_Kn}), and its orientation reversal $-L(9,8)=L(9,1)$ bounds the negative definite 4-manifold $X(9,1)$. Up to an automorphism of $-\mathbb{Z}^2$, it is evident that there is a unique embedding $Q_{X(9,1)} \hookrightarrow -\mathbb{Z}^2$, as indicated in Figure \ref{fig:embedding_A8}. The orthogonal complement is generated by $e_2$, which has square $-1\neq -9$. Hence, the condition given by Corollary \ref{cor:orthogonal_complement} is not satisfied.

 \begin{figure}[!th]
    \centering
\begin{tikzpicture}[scale=1]
\draw (0,-1.5) node[circle, fill, inner sep=1.2pt, black]{};
\draw (0,-1.5) node[above]{$-9$};
\draw (0,-1.5) node[below]{$3e_1$};
\end{tikzpicture}
\caption{An embedding of $Q_{X(9,1)}$ into $-\mathbb{Z}^2$.}
\label{fig:embedding_A8}
\end{figure}

(2) Type $A_7$: The link of a singularity of type $A_7$ is the lens space $L(8,7)$, and its orientation reversal $-L(8,7)=L(8,1)$ bounds the negative definite 4-manifold $X(8,1)$. Up to an automorphism of $-\mathbb{Z}^2$, there is a unique embedding $Q_{X(8,1)} \hookrightarrow -\mathbb{Z}^2$, as shown in Figure \ref{fig:embedding_A7}. The orthogonal complement is generated by $e_1-e_2$, which has square $-2\neq -8$. 
 \begin{figure}[!th]
    \centering
\begin{tikzpicture}[scale=1]
\draw (0,-1.5) node[circle, fill, inner sep=1.2pt, black]{};
\draw (0,-1.5) node[above]{$-8$};
\draw (0,-1.5) node[below]{$2e_1+2e_2$};
\end{tikzpicture}
\caption{An embedding of $Q_{X(8,1)}$ into $-\mathbb{Z}^2$.}
\label{fig:embedding_A7}
\end{figure}

\end{proof}

Recall that the remaining 7 types $E_8$, $E_7$, $E_6$, $D_5$, $A_4$, $A_2A_1$, $A_1$ can be realized by a $\mathbb{Q}$-homology $\mathbb{CP}^2$ whose smooth locus is simply-connected \cite[Lemma 6]{Miyanishi-Zhang-1988}. This concludes the proofs of Theorem \ref{thm:Gorenstein_H1=0} and Corollary \ref{cor:Gorenstein_pi1=1}.

\section{The Index Two Case (Proof of Theorem \ref{thm:index2_list_H1=0})}\label{sec:index_2_case}
Let $S$ be a $\mathbb{Q}$-homology $\mathbb{CP}^2$ of index $2$, and let $f\colon\tilde{S}\to S$ be the minimal resolution of $S$. Then, up to numerical equivalence, we have \[
K_{\tilde{S}}= f^*K_S - \sum_{p\in \textrm{Sing}(S)} D_p,
\]
where $D_p$ is an effective $\mathbb{Q}$-divisor supported on $f^{-1}(p)$. It follows that \[
K_S^2 = K_{\tilde{S}}^2- \sum_{p\in \textrm{Sing}(S)} D_p^2.
\]
Note that \[
K_{\tilde{S}}^2= 9-L
\]
where $L$ is the number of exceptional curves of $f\colon\tilde{S}\to S$ \cite[Corollary 3.4]{Hwang-Keum-2011-2}, and that \[
D_p^2 = \begin{cases}
    0, &\textrm{if $p$ is a rational double point,} \\ 
    -1, & \textrm{if $p$ is of type $K_n$}
\end{cases}
\]
(see \cite[Lemma 3.6]{Hwang-Keum-2011-2}).
Therefore, \begin{equation}\label{eq:K_index_2}
    K_{S}^2 = 9 + (\textrm{$\#$ of singularities of index 2}) - L.
\end{equation}

Now suppose that $H_1(S^0;\mathbb{Z})=0$. Then the numbers $|\det(R_p)|$ are pairwise relatively prime, and $K_S$ is not numerically trivial (Lemma \ref{lem:canonical_divisor}(1)). It follows that $S$ has exactly one $K_n$ singularity and $K_S^2=10-L>0$ by \eqref{eq:K_index_2}. Also, $S$ cannot have a singularity of type $A_n$ (for odd $n$), type $D_n$, or type $E_7$ (see Table \ref{tab:|det|}). It easily follows that the singularity type of $S$ is one of the following 28 types: 
\begin{align*}
  &  K_9,\quad K_8,\quad K_7,\quad K_7A_2,\quad K_6,\quad K_5,\quad K_5A_2,\quad K_4,\quad K_4A_2,\quad K_4A_4, \\ &
K_3,\quad K_3A_4, \quad K_3A_6,\quad K_2,\quad K_2A_2,\quad K_2A_4,\quad K_2A_6, \quad  K_2E_6,\quad 
K_2A_2A_4, \\ &  K_1,\quad K_1A_2,\quad K_1A_4, \quad K_1A_6,\quad K_1A_8,\quad K_1E_6,\quad K_1E_8,\quad K_1A_2A_4,\quad K_1A_2A_6.
\end{align*}

Among these $28$ types, we first eliminate $18$ using Lemma \ref{lem:canonical_divisor}(2).

\begin{lemma}\label{lem:D_calculation} The $18$ types listed in Table \ref{tab:D_calculation} do not occur as the singularity type of $S$.
\end{lemma}
\begin{proof} By Lemma \ref{lem:canonical_divisor}(2), \[
D:=K_S^2 \cdot \prod_{p\in \textup{Sing}(S)}|\det(R_p)|
\]
must be a nonzero square number. Recall that $K_S^2=10-L$. See Table \ref{tab:D_calculation}.   
\end{proof}

\begin{table}[t]
    \centering
    \begin{tabular}{c|c || c|c}
      Type  &  $D$ &  Type  &  $D$\\ [-1em] &&&\\ 
      \hline &&&\\ [-1em]
       $K_7$ & $2^2\cdot 3\cdot 7$ & $K_3A_6$ & $2^2\cdot 3\cdot 7$\\ [-1em] &&&\\
       $K_7A_2$ & $2^2\cdot 3\cdot 7$ & $K_2A_4$ & $2^5\cdot 5$ \\ [-1em] &&&\\
       $K_6$ & $2^5\cdot 3$ & $K_2A_6$ & $2^4\cdot 7$\\ [-1em] &&&\\
       $K_5A_2$ & $2^2\cdot 3^2\cdot 5$ &$K_2E_6$ &  $2^4\cdot 3$\\ [-1em]  &&&\\
       $K_4$ & $2^5\cdot 3$ & $K_2A_2A_4$ & $2^4\cdot 3\cdot 5$\\ [-1em] &&&\\
       $K_4A_2$ & $2^6\cdot 3$ &$K_1A_2$ & $2^3\cdot 3\cdot 7$ \\ [-1em] &&&\\
       $K_4A_4$ & $2^5\cdot 5$ &$K_1A_6$ & $2^2\cdot 3\cdot 7$\\ [-1em] &&&\\
       $K_3$ & $2^2\cdot 3\cdot 7$ & $K_1A_2A_4$ & $2^2\cdot 3^2\cdot 5$ \\ [-1em]&&&\\
       $K_3A_4$ & $2^2\cdot 3^2\cdot 5$ &$K_1A_2A_6$   & $2^2\cdot 3\cdot 7$
    \end{tabular}
    \caption{Computation of $D$ for the $18$ types in Lemma \ref{lem:D_calculation}.}
    \label{tab:D_calculation}
\end{table}

\begin{lemma} The type $K_2$ does not occur as the singularity type of $S$.
\end{lemma}
\begin{proof} Suppose $S$ is a $\mathbb{Q}$-homology $\mathbb{CP}^2$ with singularity type $K_2$. Then $K_S^2=8>0$ by (\ref{eq:K_index_2}), so $K_S$ is either ample or anti-ample. However, $K_S$ cannot be anti-ample by Theorem \ref{thm:Alexeev-Nikulin}. Thus, $K_S$ is ample, and $S$ must satisfy the orbifold BMY inequality (Theorem \ref{thm:oBMY}(1)). However, \[
3e_{\textup{orb}}(S)= 3\cdot \dfrac{17}{8} =\frac{51}{8} < 8=K_S^2,
\]
which is a contradiction. \end{proof}

\begin{lemma}\label{lem:index2_complement} The types $K_9$, $K_8$, and $K_1A_8$ do not occur as the singularity type of $S$.     
\end{lemma}
\begin{proof} We analyze the condition from Corollary \ref{cor:orthogonal_complement} in each case.

(1) Type $K_9$: The link of a singularity of type $K_9$ is the lens space $L(36,17)$ (Example \ref{ex:link_of_An_Kn}), and its orientation reversal $-L(36,17)=L(36,19)$ bounds the negative definite 4-manifold $X(36,19)$. Up to an automorphism of $-\mathbb{Z}^4$, there are exactly two distinct embeddings $Q_{X(36,19)} \hookrightarrow -\mathbb{Z}^4$, as shown in Figure \ref{fig:embedding_K9}. For the first embedding, the orthogonal complement is generated by $e_4$, and for the second embedding, the orthogonal complement is generated by $e_1-e_2-e_3+e_4$. Since $e_4^2=-1\neq -36$ and $(e_1-e_2-e_3+e_4)^2=-4 \neq -36$ for each embedding, the condition from Corollary \ref{cor:orthogonal_complement} is not satisfied.

      \begin{figure}[!th]
    \centering
\begin{tikzpicture}[scale=1]
\draw (-2,0) node[circle, fill, inner sep=1.2pt, black]{};
\draw (0,0) node[circle, fill, inner sep=1.2pt, black]{};
\draw (2,0) node[circle, fill, inner sep=1.2pt, black]{};

\draw (-2,0) node[above]{$-2$};
\draw (0,0) node[above]{$-10$};
\draw (2,0) node[above]{$-2$};
\draw (-2,0)--(2,0);

\draw (-2,0) node[below]{$e_1-e_2$};
\draw (0,0) node[below]{$-e_1+3e_3$};
\draw (2,0) node[below]{$e_1+e_2$};

\draw (5,0) node[circle, fill, inner sep=1.2pt, black]{};
\draw (7,0) node[circle, fill, inner sep=1.2pt, black]{};
\draw (9,0) node[circle, fill, inner sep=1.2pt, black]{};

\draw (5,0) node[above]{$-2$};
\draw (7,0) node[above]{$-10$};
\draw (9,0) node[above]{$-2$};
\draw (5,0)--(9,0);

\draw[densely dotted] (7,0)--(7,-0.5);

\draw (5,0) node[below]{$e_1+e_2$};
\draw (7,-0.5) node[below]{$-2e_1+e_2-2e_3+e_4$};
\draw (9,0) node[below]{$e_3+e_4$};
\end{tikzpicture}
\caption{Two embeddings of $Q_{X(36,19)}$ into $-\mathbb{Z}^4$.}
\label{fig:embedding_K9}
\end{figure}
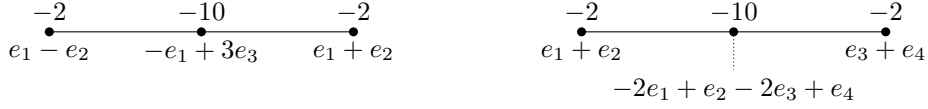

(2) Type $K_8$: The corresponding link is $L(32,15)$, and $-L(32,15)=L(32,17)$ bounds $X(32,17).$ Up to an automorphism of $-\mathbb{Z}^4$, there is a unique embedding $Q_{X(32,17)} \hookrightarrow -\mathbb{Z}^4$, as shown in Figure \ref{fig:embedding_K8}. The orthogonal complement is generated by $e_3-e_4$, which has square $-2\neq -32$.

 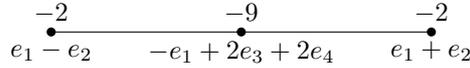
\begin{figure}[!th]
    \centering
\begin{tikzpicture}[scale=1]
\draw (-2.5,0) node[circle, fill, inner sep=1.2pt, black]{};
\draw (0,0) node[circle, fill, inner sep=1.2pt, black]{};
\draw (2.5,0) node[circle, fill, inner sep=1.2pt, black]{};

\draw (-2.5,0) node[above]{$-2$};
\draw (0,0) node[above]{$-9$};
\draw (2.5,0) node[above]{$-2$};
\draw (-2.5,0)--(2.5,0);

\draw (-2.5,0) node[below]{$e_1-e_2$};
\draw (0,0) node[below]{$-e_1+2e_3+2e_4$};
\draw (2.5,0) node[below]{$e_1+e_2$};
\end{tikzpicture}
\caption{An embedding of $Q_{X(32,15)}$ into $-\mathbb{Z}^4$.}
\label{fig:embedding_K8}
\end{figure}

(3) Type $K_1A_8$: The links corresponding to $K_1$ and $A_8$ are $L(4,1)$ and $L(9,8)$, respectively. $-L(4,1)$ and $-L(9,8)$ bound $X(4,3)$ and $X(9,1)$, respectively. Up to an automorphism of $-\mathbb{Z}^5$, there are exactly two distinct embeddings $Q_{X(4,3)}\oplus Q_{X(9,1)} \hookrightarrow -\mathbb{Z}^5$, as shown in Figure \ref{fig:embedding_K1A8}. 

For the first embedding, the orthogonal complement is generated by $e_5$ which has square $-1\neq -36$. In the second case, the orthogonal complement is generated by $e_2+e_3+e_4-e_5$, which has square $-4\neq -36$.

 \begin{figure}[!th]
    \centering
\begin{tikzpicture}[scale=1]
\draw (-2,0) node[circle, fill, inner sep=1.2pt, black]{};
\draw (0,0) node[circle, fill, inner sep=1.2pt, black]{};
\draw (2,0) node[circle, fill, inner sep=1.2pt, black]{};
\draw (0,-1.5) node[circle, fill, inner sep=1.2pt, black]{};

\draw (-2,0) node[above]{$-2$};
\draw (0,0) node[above]{$-2$};
\draw (2,0) node[above]{$-2$};
\draw (0,-1.5) node[above]{$-9$};
\draw (-2,0)--(2,0);

\draw (-2,0) node[below]{$e_1-e_2$};
\draw (0,0) node[below]{$-e_1+e_3$};
\draw (2,0) node[below]{$e_1+e_2$};
\draw (0,-1.5) node[below]{$3e_4$};

\draw (5,0) node[circle, fill, inner sep=1.2pt, black]{};
\draw (7,0) node[circle, fill, inner sep=1.2pt, black]{};
\draw (9,0) node[circle, fill, inner sep=1.2pt, black]{};
\draw (7,-1.5) node[circle, fill, inner sep=1.2pt, black]{};

\draw (5,0) node[above]{$-2$};
\draw (7,0) node[above]{$-2$};
\draw (9,0) node[above]{$-2$};
\draw (5,0)--(9,0);
\draw (7,-1.5) node[above]{$-9$};

\draw (5,0) node[below]{$-e_2+e_3$};
\draw (7,0) node[below]{$e_2-e_4$};
\draw (9,0) node[below]{$e_4+e_5$};
\draw (7,-1.5) node[below]{$3e_1$};
\end{tikzpicture}
\caption{Two embeddings of $Q_{X(4,3)}\oplus Q_{X(9,1)}$ into $-\mathbb{Z}^5$.}
\label{fig:embedding_K1A8}
\end{figure}
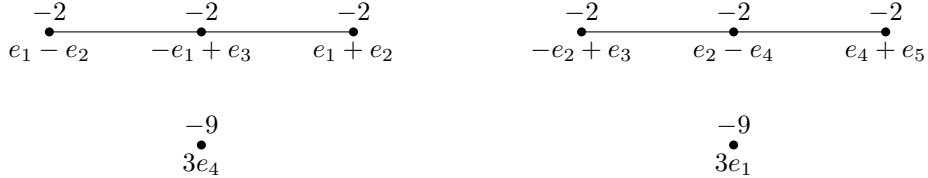
\end{proof}

\begin{lemma}\label{lem:index2_linkingform} The type $K_1E_6$ does not occur as the singularity type of $S$.
\end{lemma}
\begin{proof} Suppose $S$ is a $\mathbb{Q}$-homology $\mathbb{CP}^2$ with singularity type $K_1E_6$ such that $H_1(S^0;\mathbb{Z})=0$. Consider the smooth 4-manifold $Z^0$ constructed in Section \ref{subsec:top_and_smooth_obstructions}, whose boundary is \[
    \partial Z^0=-L(4,1)\# -L(E_6)=L(4,3)\# -L(E_6).\]
The smooth 4-manifold $Z^0$ has intersection form $(12)$ by Corollary \ref{cor:linking_form}, so the linking form on $H_1(\partial Z^0;\mathbb{Z})=\mathbb{Z}_{12}$ is represented by $\left(-\dfrac{1}{12}\right)$ by Lemma \ref{lem:linking_form}(1). Note that symmetric bilinear $\mathbb{Q}/\mathbb{Z}$-valued forms on $\mathbb{Z}_{12}$ isomorphic to $\left(-\dfrac{1}{12}\right)$ are precisely those of the form $\left(-\dfrac{a^2}{12}\right)$, where $(a,12)=1$.

On the other hand, recall that $L(4,3)$ is obtained from $S^3$ by $\left(-{4}/{3}\right)$-surgery along the unknot, and that $-L(E_6)$ is obtained by $(+3)$-surgery along the right-handed trefoil knot (Example \ref{ex:d_invs_En}). By Lemma \ref{lem:linking_form}(2), the linking form on $H_1(L(4,3);\mathbb{Z})=\mathbb{Z}_4$ is represented by $\left(\displaystyle\frac{3}{4}\right)$, and the linking form on $H_1(-L(E_6);\mathbb{Z})=\mathbb{Z}_3$ is represented by $\left(-\dfrac{1}{3}\right)$.
Via the isomorphism $H_1(L(4,3)\# -L(E_6);\mathbb{Z})=\mathbb{Z}_4\oplus \mathbb{Z}_3 \cong \mathbb{Z}_{12}$, given by $(1,1)\mapsto 1$, we see that the linking form on $H_1(L(4,3)\# -L(E_6);\mathbb{Z})=\mathbb{Z}_{12}$ is represented by $\left(\displaystyle\frac{5}{12}\right) \cong \left(-\displaystyle\frac{7}{12}\right)$. However, since $7$ is not a quadratic residue modulo $12$, the form $\left(-\dfrac{7}{12}\right)$ is not isomorphic to $\left(-\dfrac{1}{12}\right)$, which leads to a contradiction. 
\end{proof}

\begin{lemma} The type $K_1E_8$ does not occur as the singularity type of $S$.
\end{lemma}
\begin{proof} Consider the corresponding links $L(K_1)=L(4,3)$ and $L(E_8)$. The link $L(K_1)$ has two spin structures with $d$-invariants $-{3}/{4}$ and ${1}/{4}$ (Example \ref{ex:d_invs_An_Kn}), while $L(E_8)$ has a unique spin$^c$ structure, which is induced by the unique spin structure, with the $d$-invariant of $2$ (Example \ref{ex:d_invs_En}). Hence, $d$-invariants of $L(K_1)\#L(E_8)$ cannot satisfy the constraint given in Corollary \ref{cor:spin_d_2}. \end{proof}

Recall that the remaining 4 types $K_5$, $K_2A_2$, $K_1$, and $K_1A_4$ can be realized as a $\mathbb{Q}$-homology $\mathbb{CP}^2$ whose smooth locus is simply-connected (Theorem \ref{thm:Kojima}). This concludes the proofs of Theorem \ref{thm:index2_list_H1=0} and Corollary \ref{cor:index2_list_pi1=1}.

\section{The Index Three Case (Proof of Theorem \ref{thm:index3_list})}\label{sec:index3_case}
Recall that a quotient singularity whose germ is locally analytically isomorphic to $(\mathbb{C}^2/G,0)$ is of index 3 if and only if $[G:G\cap \textup{SL}(2,\mathbb{C})]=3$, where $G\subset \textup{GL}(2,\mathbb{C})$ is a finite subgroup which does not contain any reflections. Following Brieskorn's notation \cite[Section 2]{Brieskorn-1968}, such a group $G$ is one of the following \cite[Section 2]{Ohashi-Taki-2012}: 

\begin{enumerate}[label=\textup{(\arabic*)}]
\item $G=C_{n,q}$, where $0<q<n$, $(n,q)=1$, and $\displaystyle\frac{1+q}{n}\in \frac{1}{3}\mathbb{Z}\setminus \mathbb{Z}$, 
\item $G=(\mathbb{Z}_6,\mathbb{Z}_6;\mathbb{D}_2,\mathbb{D}_2)$,
\item $G=(\mathbb{Z}_6,\mathbb{Z}_6;\mathbb{D}_n,\mathbb{D}_n)$, where $n\not\equiv 0 \mod 3$ and $n\geq 4$. 
\end{enumerate}
For the cyclic case (1), the weighted dual graph of the reduced exceptional divisor of a minimal resolution of the corresponding singularity is one of those listed in Table \ref{tab:singularities_of_index_3}. For the non-cyclic cases (2) and (3), the weighted dual graph is one of those listed in Table \ref{tab:noncyclic_singularities_of_index_3} \cite[Table 1]{Ohashi-Taki-2012}. For singularities corresponding to cases (2) and (3), we have $|\det(R_p)|=12$ (see \cite[Lemma 3.7]{Hwang-Keum-2011-2}).

\begin{table}[t]
    \centering
    \begin{tabular}{c|c|c}
      Type  &  Dual graph &  Link  \\ [-1em] &&\\  
      \hline 
       $A_1(1)$ &  \begin{tikzpicture}[baseline=0]
\draw (0,0) node[circle, fill, inner sep=1.2pt, black]{};
\draw (0,0) node[above]{$-3$};
\end{tikzpicture}
& $L(3,1)$ \\ [-0.8em] &&\\
$A_1(2)$ & \begin{tikzpicture}[baseline=0]
\draw (0,0) node[circle, fill, inner sep=1.2pt, black]{};
\draw (0,0) node[above]{$-6$};
\end{tikzpicture} & $L(6,1)$  \\ [-0.8em] &&\\
      \hline $A_2(1,2)$ & \begin{tikzpicture}[baseline=0]
\draw (0,0) node[circle, fill, inner sep=1.2pt, black]{};
\draw (1,0) node[circle, fill, inner sep=1.2pt, black]{};
\draw (0,0) node[above]{$-2$};
\draw (1,0) node[above]{$-5$};
\draw (0,0)--(1,0); 
\end{tikzpicture} & $L(9,5)$  \\ [-0.8em] &&\\
      \hline
      $A_3(1,1)$ &  \begin{tikzpicture}[baseline=0]
\draw (0,0) node[circle, fill, inner sep=1.2pt, black]{};
\draw (1,0) node[circle, fill, inner sep=1.2pt, black]{};
\draw (2,0) node[circle, fill, inner sep=1.2pt, black]{};
\draw (0,0) node[above]{$-2$};
\draw (1,0) node[above]{$-4$};
\draw (2,0) node[above]{$-2$};
\draw (0,0)--(2,0); 
\end{tikzpicture} & $L(12,7)$ \\ [-0.8em] &&\\
      \hline
      $A_n(1,1)$ $(n\geq 4)$ & \begin{tikzpicture}[baseline=0]
\draw (-3,0) node[circle, fill, inner sep=1.2pt, black]{};
\draw (-2,0) node[circle, fill, inner sep=1.2pt, black]{};
\draw (-1,0) node[circle, fill, inner sep=1.2pt, black]{};
\draw (1,0) node[circle, fill, inner sep=1.2pt, black]{};
\draw (2,0) node[circle, fill, inner sep=1.2pt, black]{};
\draw (3,0) node[circle, fill, inner sep=1.2pt, black]{};
\draw (-3,0) node[above]{$-2$};
\draw (-2,0) node[above]{$-3$};
\draw (-1,0) node[above]{$-2$};
\draw (1,0) node[above]{$-2$};
\draw (2,0) node[above]{$-3$};
\draw (3,0) node[above]{$-2$};
\draw (0,0) node{$\cdots$};
\draw (-3,0)--(-1,0) (-1,0)--(-0.5,0) (0.5,0)--(1,0)  (1,0)--(3,0) ;
\draw [thick,decorate,decoration={mirror,brace,amplitude=3pt},xshift=0pt]
	(-1,-0.2) -- (1,-0.2) node [black,midway,xshift=0pt,yshift=-10pt] 
	{$n-4$};
\end{tikzpicture} & $L(9n-15,6n-11)$ \\ [-0.8em] &&\\
$A_n(1,2)$ $(n\geq 3)$ & \begin{tikzpicture}[baseline=0]
\draw (-3,0) node[circle, fill, inner sep=1.2pt, black]{};
\draw (-2,0) node[circle, fill, inner sep=1.2pt, black]{};
\draw (-1,0) node[circle, fill, inner sep=1.2pt, black]{};
\draw (1,0) node[circle, fill, inner sep=1.2pt, black]{};
\draw (2,0) node[circle, fill, inner sep=1.2pt, black]{};
\draw (-3,0) node[above]{$-2$};
\draw (-2,0) node[above]{$-3$};
\draw (-1,0) node[above]{$-2$};
\draw (1,0) node[above]{$-2$};
\draw (2,0) node[above]{$-4$};
\draw (0,0) node{$\cdots$};
\draw (-3,0)--(-1,0) (-1,0)--(-0.5,0) (0.5,0)--(1,0)  (1,0)--(2,0) ;
\draw [thick,decorate,decoration={mirror,brace,amplitude=3pt},xshift=0pt]
	(-1,-0.2) -- (1,-0.2) node [black,midway,xshift=0pt,yshift=-10pt] 
	{$n-3$};
\end{tikzpicture} & $L(9n-9,6n-7)$ \\ [-0.8em] &&\\
$A_n(2,2)$ $(n\geq 2)$ & \begin{tikzpicture}[baseline=0]
\draw (-2,0) node[circle, fill, inner sep=1.2pt, black]{};
\draw (-1,0) node[circle, fill, inner sep=1.2pt, black]{};
\draw (1,0) node[circle, fill, inner sep=1.2pt, black]{};
\draw (2,0) node[circle, fill, inner sep=1.2pt, black]{};
\draw (-2,0) node[above]{$-4$};
\draw (-1,0) node[above]{$-2$};
\draw (1,0) node[above]{$-2$};
\draw (2,0) node[above]{$-4$};
\draw (0,0) node{$\cdots$};
\draw (-2,0)--(-1,0) (-1,0)--(-0.5,0) (0.5,0)--(1,0)  (1,0)--(2,0) ;
\draw [thick,decorate,decoration={mirror,brace,amplitude=3pt},xshift=0pt]
	(-1,-0.2) -- (1,-0.2) node [black,midway,xshift=0pt,yshift=-10pt] 
	{$n-2$};
\end{tikzpicture} & $L(9n-3,3n-2)$
    \end{tabular}
    \caption{Cyclic quotient singularities of index three.}   \label{tab:singularities_of_index_3}
\end{table}

\subsection{Obstructions}
Let $S$ be a $\mathbb{Q}$-homology $\mathbb{CP}^2$ of index 3 with $H_1(S^0;\mathbb{Z})=0$. Then for $p\in\textup{Sing}(S)$, the numbers $|\det(R_p)|=|H_1(L_p;\mathbb{Z})|$ are pairwise relatively prime (Lemma \ref{lem:canonical_divisor}(1)). Therefore, $S$ must have exactly one singularity of index 3, and the other singularities of $S$ must be rational double points.

\begin{table}[t]
    \centering
    \begin{tabular}{c|c}
      Type  &  Dual graph  \\ [-1em] &\\  
      \hline 
       $D_4(1)$ &   \begin{tikzpicture}[baseline=0]
\draw (0,0) node[circle, fill, inner sep=1.2pt, black]{};
\draw (1,0) node[circle, fill, inner sep=1.2pt, black]{};
\draw (2,0) node[circle, fill, inner sep=1.2pt, black]{};
\draw (1,-0.7) node[circle, fill, inner sep=1.2pt, black]{};
\draw (0,0) node[above]{$-2$};
\draw (1,0) node[above]{$-3$};
\draw (2,0) node[above]{$-2$};
\draw (1,-0.7) node[left]{$-2$};
\draw (0,0)--(2,0) (1,0)--(1,-0.7); 
\end{tikzpicture}  \\ [-0.8em] &\\  
      \hline
      $D_n(1)$ $(n\geq 5)$ & \begin{tikzpicture}[baseline=0]
\draw (-3,0) node[circle, fill, inner sep=1.2pt, black]{};
\draw (-2,0) node[circle, fill, inner sep=1.2pt, black]{};
\draw (-1,0) node[circle, fill, inner sep=1.2pt, black]{};
\draw (1,0) node[circle, fill, inner sep=1.2pt, black]{};
\draw (2,0) node[circle, fill, inner sep=1.2pt, black]{};
\draw (1,-0.8) node[circle, fill, inner sep=1.2pt, black]{};
\draw (-3,0) node[above]{$-2$};
\draw (-2,0) node[above]{$-3$};
\draw (-1,0) node[above]{$-2$};
\draw (1,0) node[above]{$-2$};
\draw (2,0) node[above]{$-2$};
\draw (1,-0.8) node[right]{$-2$};
\draw (0,0) node{$\cdots$};
\draw (-3,0)--(-1,0) (-1,0)--(-0.5,0) (0.5,0)--(1,0)  (1,0)--(2,0) (1,0)--(1,-0.8);
\draw [thick,decorate,decoration={mirror,brace,amplitude=3pt},xshift=0pt]
	(-1,-0.15) -- (0.9,-0.15) node [black,midway,xshift=0pt,yshift=-8pt] 
	{$n-4$};
\end{tikzpicture}  \\ [-0.8em] &\\
$D_n(2)$ $(n\geq 4)$ & \begin{tikzpicture}[baseline=0]
\draw (-2,0) node[circle, fill, inner sep=1.2pt, black]{};
\draw (-1,0) node[circle, fill, inner sep=1.2pt, black]{};
\draw (1,0) node[circle, fill, inner sep=1.2pt, black]{};
\draw (1,-0.8) node[circle, fill, inner sep=1.2pt, black]{};
\draw (2,0) node[circle, fill, inner sep=1.2pt, black]{};
\draw (-2,0) node[above]{$-4$};
\draw (-1,0) node[above]{$-2$};
\draw (1,0) node[above]{$-2$};
\draw (1,-0.8) node[right]{$-2$};
\draw (2,0) node[above]{$-2$};
\draw (0,0) node{$\cdots$};
\draw (-2,0)--(-1,0) (-1,0)--(-0.5,0) (0.5,0)--(1,0)  (1,0)--(2,0) (1,0)--(1,-0.8);
\draw [thick,decorate,decoration={mirror,brace,amplitude=3pt},xshift=0pt]
	(-1,-0.15) -- (0.9,-0.15) node [black,midway,xshift=0pt,yshift=-8pt] 
	{$n-3$};
\end{tikzpicture} 
    \end{tabular}
    \caption{Non-cyclic quotient singularities of index three.}   \label{tab:noncyclic_singularities_of_index_3}
\end{table}

Let $f\colon\tilde{S}\to S$ be the minimal resolution of $S$. As in Section \ref{sec:index_2_case}, we have \[
K_S^2=K_{\tilde{S}}^2-\sum_{p\in \textup{Sing}(S)} D_p^2.
\]
Also, $K_{\tilde{S}}^2=9-L$, where $L$ is the number of exceptional curves. From \cite[Lemma 3.6, Lemma 3.7]{Hwang-Keum-2011-2}, we have \[
D_p^2=\begin{cases}
    -\frac{1}{3}, & \textrm{if $p$ is of type $A_1(1)$}, \\
    -\frac{4}{3}, & \textrm{if $p$ is of type $A_n(1,1)$ $(n\geq 3)$}, \\
    -2, & \textrm{if $p$ is of type $A_n(1,2)$ $(n\geq 2)$}, \\
    -\frac{8}{3}, & \textrm{if $p$ is of type $A_1(2)$ or $A_n(2,2)$}, \\
    -\frac{2}{3}, & \textrm{if $p$ is of type $D_n(1)$ $(n\geq 5)$}, \\
    -\frac{4}{3}, & \textrm{if $p$ is of type $D_n(2)$ $(n\geq 5)$}.
\end{cases}
\]
Note that $D_p^2=0$ if $p$ is a rational double point.

\medskip

\textbf{\hypertarget{Case_1}{Case 1}:} Suppose that the unique singularity of index $3$ in $S$ is of type $A_1(1)$. In this case, $K_S^2=9+\frac{1}{3}-L>0$, and in particular, $L\leq 9$. Since the numbers $|\det(R_p)|$ are pairwise relatively prime, $S$ cannot have a singularity of type $A_2$, $A_5$, $A_8$, or $E_6$. Moreover, since $H_1(L(D_n);\mathbb{Z})=\mathbb{Z}_2\oplus\mathbb{Z}_2$ is not cyclic, $S$ cannot have a singularity of type $D_n$ with even $n$ (see Corollary \ref{cor:linking_form} and Example \ref{ex:d_invs_Dn}). It follows that the singularity type of $S$ is one of the 13 types listed in Table \ref{tab:D_index3_case1}. Only the following 6 types satisfy the condition of Lemma \ref{lem:canonical_divisor}(2): \[
A_1(1)E_8,\quad A_1(1)D_7,\quad A_1(1)A_6,\quad A_1(1)A_4A_1,\quad A_1(1)A_3,\quad A_1(1).\] 
Each of these $6$ cases can be realized by a $\mathbb{Q}$-homology $\mathbb{CP}^2$ whose smooth locus is simply-connected; see No.1, 18a, 18b, 66-68, and 97 in \cite[Appendix]{Zhang-1989}.

\begin{table}[t]
    \centering
    \begin{tabular}{c|c || c|c || c|c|| c|c}
      Type  &  $D$ &  Type  &  $D$ & Type  &  $D$ &  Type  &  $D$ \\ [-1em] &&&&&&&\\ 
      \hline &&&&&&&\\ [-1em]
       $\pmb{A_1(1)E_8}$ & $1$ & $A_1(1)A_7$ & $2^5$ & $\pmb{A_1(1)A_4A_1}$ & $2^2\cdot 5^2$ & $\pmb{A_1(1)}$ & $5^2$\\ [-1em] &&&&&\\
       $A_1(1)E_7$  & $2^3$ & $A_1(1)A_6A_1$ & $2^3\cdot 7$ & $A_1(1)A_4$ & $5\cdot 13$   \\[-1em] &&&&&\\
       $\pmb{A_1(1)D_7}$ & $2^4$ & $\pmb{A_1(1)A_6}$ & $7^2$ & $\pmb{A_1(1)A_3}$ & $2^6$ \\[-1em] &&&&&\\
       $A_1(1)D_5$ & $2^3\cdot 5$ & $A_1(1)A_4A_3$ & $2^4\cdot 5$ & $A_1(1)A_1$ & $2^2\cdot 11$
    \end{tabular}
    \caption{Computation of $D$ for types in \protect\hyperlink{Case_1}{Case 1}.}
    \label{tab:D_index3_case1}
\end{table}

\medskip

\textbf{\hypertarget{Case_2}{Case 2}:} Suppose that the unique singularity of index 3 in $S$ is of type $A_n(1,1)$ ($n\geq 3$). In this case, $K_S^2=9+\frac{4}{3}-L>0$, so $L\leq 10$. Noting that $S$ cannot have a singularity of type $D_n$ with $n$ even, and that the numbers $|\det(R_p)|$ are pairwise relatively prime, we conclude that the singularity type of $S$ is one of the 19 types listed in Table \ref{tab:D_index3_case2}. Among these, $D$ is a square number only for the two types $A_6(1,1)$ and $A_{10}(1,1)$.

\begin{table}[t]
    \centering
    \begin{tabular}{c|c || c|c || c|c|| c|c}
      Type  &  $D$ &  Type  &  $D$ & Type  &  $D$ &  Type  &  $D$ \\ [-1em] &&&&&&&\\ 
      \hline &&&&&&&\\ [-1em]
       $A_3(1,1)A_6$ & $2^4\cdot 7$ & $A_4(1,1)A_4$ & $5\cdot 7^2$ & $A_6(1,1)A_4$ & $5\cdot 13$ & $A_8(1,1)A_1$ & $2^3\cdot 19$\\ [-1em] &&&&&&\\
       $A_3(1,1)A_4$  & $2^3\cdot 5^2$ & $A_4(1,1)A_3$ & $2^3\cdot 5\cdot 7$ & $A_6(1,1)A_3$ & $2^4\cdot 13$  & 
       $A_8(1,1)$ & $7\cdot 19$\\[-1em] &&&&&&\\
       $A_3(1,1)$ & $2^3\cdot 11$ & $A_4(1,1)A_1$ & $2^5\cdot 7$ & $A_6(1,1)A_1$ & $2^2\cdot 5\cdot 13$ & $A_9(1,1)$ & $2^3\cdot 11$\\[-1em] &&&&&&\\
       $A_4(1,1)D_5$ & $2^4\cdot 7$ & $A_4(1,1)$ & $7\cdot 19$ & $\pmb{A_6(1,1)}$ & $13^2$ & $\pmb{A_{10}(1,1)}$ & $5^2$ \\[-1em] &&&&&\\
       $A_4(1,1)A_4A_1$ & $2^3\cdot 5\cdot 7$ & $A_5(1,1)$ & $2^5\cdot 5$ & $A_7(1,1)$ & $2^5\cdot 5$
    \end{tabular}
    \caption{Computation of $D$ for types in \protect\hyperlink{Case_2}{Case 2}.}
    \label{tab:D_index3_case2}
\end{table}

\medskip

\textbf{\hypertarget{Case_3}{Case 3}:} Suppose that the unique singularity of index 3 in $S$ is of type $A_n(1,2)$ ($n\geq 2$). In this case, $K_S^2=9+2-L=11-L>0$, so $L\leq 10$. The singularity type of $S$ is one of the $33$ types listed in Table \ref{tab:D_index3_case3}. Among these, $D$ is a square number only for the following $12$ types: 
\[
A_2(1,2)E_8, \quad A_2(1,2)E_7, \quad A_2(1,2)D_5, \quad A_2(1,2)A_7, \quad A_2(1,2)A_4, \quad A_2(1,2)A_1, \quad A_2(1,2),
\]  
\[
A_3(1,2), \quad A_4(1,2)A_1, \quad A_6(1,2), \quad A_9(1,2), \quad A_{10}(1,2).
\]

\begin{table}[t]
    \centering
    \scalebox{0.9}{
    \begin{tabular}{c|c || c|c || c|c|| c|c}
      Type  &  $D$ &  Type  &  $D$ & Type  &  $D$ &  Type  &  $D$ \\ [-1em] &&&&&&&\\ 
      \hline &&&&&&&\\ [-1em]
       $A_2(1,2)E_8$ & $3^2$ & $\pmb{A_2(1,2)A_4}$ & $3^2\cdot 5^2$ & $A_4(1,2)A_4A_1$ & $2^2\cdot 3^3\cdot 5$ & $A_6(1,2)$ & $3^2\cdot 5^2$\\ [-1em] &&&&&&\\
       $\pmb{A_2(1,2)E_7}$  & $2^2\cdot 3^2$ & $A_2(1,2)A_3$ & $2^3\cdot 3^3$ & $A_4(1,2)A_4$ & $3^4\cdot 5$  & 
       $A_7(1,2)$ & $2^3\cdot 3^3$\\[-1em] &&&&&&\\
       $A_2(1,2)D_7$ & $2^3\cdot 3^2$ & $\pmb{A_2(1,2)A_1}$ & $2^4\cdot 3^2$ & $A_4(1,2)A_3$ & $2^4\cdot 3^3$ & $A_8(1,2)A_1$ & $2^2\cdot 3^2\cdot 7$\\[-1em] &&&&&&\\
       $A_2(1,2)D_5$ & $2^4\cdot 3^2$ & $A_2(1,2)$ & $3^4$ & $\pmb{A_4(1,2)A_1}$ & $2^2\cdot 3^4$ & $A_8(1,2)$ & $3^3\cdot 7$ \\[-1em] &&&&&&\\
       $A_2(1,2)A_7$ & $2^4\cdot 3^2$ & $A_3(1,2)A_6$ & $2^2\cdot 3^2\cdot 7$ & $A_4(1,2)$ & $3^3\cdot 7$ & $A_9(1,2)$ & $2^4\cdot 3^2$ \\[-1em] &&&&&&\\
       $A_2(1,2)A_6A_1$ & $2^2\cdot 3^2\cdot 7$ & $A_3(1,2)A_4$ & $2^3\cdot 3^2\cdot 5$ & $A_5(1,2)A_4$ & $2^3\cdot 3^2\cdot 5$ & $A_{10}(1,2) $& $3^4$\\[-1em] &&&&&&\\
       $A_2(1,2)A_6$ & $3^3\cdot 7$ & $A_3(1,2)$ & $2^4\cdot 3^2$ & $A_5(1,2)$ & $2^3\cdot 3^3$  \\[-1em] &&&&&\\
       $A_2(1,2)A_4A_3$ & $2^3\cdot 3^2\cdot 5$ & $A_4(1,2)D_5$ & $2^3\cdot 3^3$ & $A_6(1,2)A_3$ & $2^3\cdot 3^2\cdot 5$  \\[-1em] &&&&&\\
       $A_2(1,2)A_4A_1$ & $2^3\cdot 3^2\cdot 5$ & $A_4(1,2)A_6$ & $3^3\cdot 7$ & $A_6(1,2)A_1$ & $2^3\cdot 3^2\cdot 5$ 
    \end{tabular}}
    \caption{Computation of 
    $D$ for types in \protect\hyperlink{Case_3}{Case 3}.}
    \label{tab:D_index3_case3}
\end{table}

\begin{lemma}\label{lem:index3_Donaldson} The types $A_2(1,2)A_7$, $A_2(1,2)$, $A_3(1,2)$, $A_6(1,2)$, $A_9(1,2)$, and $A_{10}(1,2)$ do not occur as the singularity type of $S$.     
\end{lemma}
\begin{proof} We apply the lattice embedding condition from Corollary \ref{cor:orthogonal_complement} for each case (cf. Lemma \ref{lem:index2_complement}). Note that $L(9n-9,6n-7)$ is the link of a singularity of type $A_n(1,2)$ $(n\geq 2)$, and its orientation reversal $-L(9n-9,6n-7)=L(9n-9,3n-2)$ bounds the negative definite 4-manifold $X(9n-9,3n-2)$. The Hirzebruch-Jung continued fraction of $\frac{9n-9}{3n-2}$ is given by $[3,n,2,2]$.

(1) Type $A_2(1,2)A_7$: There are exactly two distinct embeddings $Q_{X(9,4)}\oplus Q_{X(8,1)} \hookrightarrow -\mathbb{Z}^6$ up to an automorphism of $-\mathbb{Z}^6$, as shown in Figure \ref{fig:embedding_A_2(1,2)A_7}. For the first embedding, the orthogonal complement is generated by $e_5-e_6$, which has square $-2\neq -72$. For the second embedding, the orthogonal complement is generated by $e_1+e_2+2e_3+2e_4+2e_5-2e_6$, which has square $-18\neq -72$.

 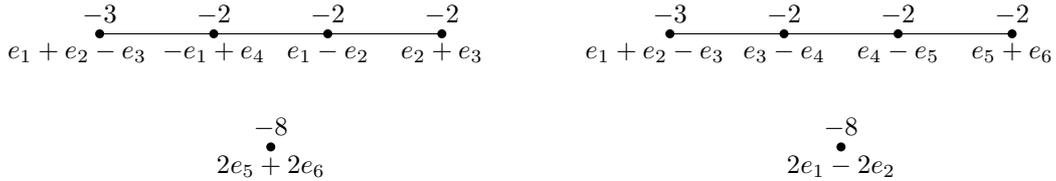
\begin{figure}[!th]
    \centering
\begin{tikzpicture}[scale=1]
\draw (-3,0) node[circle, fill, inner sep=1.2pt, black]{};
\draw (-1.5,0) node[circle, fill, inner sep=1.2pt, black]{};
\draw (0,0) node[circle, fill, inner sep=1.2pt, black]{};
\draw (1.5,0) node[circle, fill, inner sep=1.2pt, black]{};
\draw (-0.75,-1.5) node[circle, fill, inner sep=1.2pt, black]{};

\draw (-3,0) node[above]{$-3$};
\draw (-1.5,0) node[above]{$-2$};
\draw (0,0) node[above]{$-2$};
\draw (1.5,0) node[above]{$-2$};
\draw (-0.75,-1.5) node[above]{$-8$};
\draw (-3,0)--(1.5,0);

\draw (-3.3,0) node[below]{$e_1+e_2-e_3$};
\draw (-1.5,0) node[below]{$-e_1+e_4$};
\draw (0,0) node[below]{$e_1-e_2$};
\draw (1.5,0) node[below]{$e_2+e_3$};
\draw (-0.75,-1.5) node[below]{$2e_5+2e_6$};

\draw (4.5,0) node[circle, fill, inner sep=1.2pt, black]{};
\draw (6,0) node[circle, fill, inner sep=1.2pt, black]{};
\draw (7.5,0) node[circle, fill, inner sep=1.2pt, black]{};
\draw (9,0) node[circle, fill, inner sep=1.2pt, black]{};
\draw (6.75,-1.5) node[circle, fill, inner sep=1.2pt, black]{};

\draw (4.5,0) node[above]{$-3$};
\draw (6,0) node[above]{$-2$};
\draw (7.5,0) node[above]{$-2$};
\draw (9,0) node[above]{$-2$};
\draw (4.5,0)--(9,0);
\draw (6.75,-1.5) node[above]{$-8$};

\draw (4.3,0) node[below]{$e_1+e_2-e_3$};
\draw (6,0) node[below]{$e_3-e_4$};
\draw (7.5,0) node[below]{$e_4-e_5$};
\draw (9,0) node[below]{$e_5+e_6$};
\draw (6.75,-1.5) node[below]{$2e_1-2e_2$};
\end{tikzpicture}
\caption{Two embeddings of $Q_{X(9,4)}\oplus Q_{X(8,1)}$ into $-\mathbb{Z}^6$.}
\label{fig:embedding_A_2(1,2)A_7}
\end{figure}

(2) Type $A_2(1,2)$: Up to an automorphism of $-\mathbb{Z}^5$, there is a unique embedding $Q_{X(9,4)}\hookrightarrow -\mathbb{Z}^5$, as shown in Figure \ref{fig:embedding_A_2(1,2)}. The orthogonal complement is generated by $e_5$, which has square $-1\neq -9$.

\begin{figure}[!th]
    \centering
\begin{tikzpicture}[scale=1]
\draw (-3,0) node[circle, fill, inner sep=1.2pt, black]{};
\draw (-1.5,0) node[circle, fill, inner sep=1.2pt, black]{};
\draw (0,0) node[circle, fill, inner sep=1.2pt, black]{};
\draw (1.5,0) node[circle, fill, inner sep=1.2pt, black]{};

\draw (-3,0) node[above]{$-3$};
\draw (-1.5,0) node[above]{$-2$};
\draw (0,0) node[above]{$-2$};
\draw (1.5,0) node[above]{$-2$};
\draw (-3,0)--(1.5,0);

\draw (-3.3,0) node[below]{$e_1+e_2-e_3$};
\draw (-1.5,0) node[below]{$-e_1+e_4$};
\draw (0,0) node[below]{$e_1-e_2$};
\draw (1.5,0) node[below]{$e_2+e_3$};
\end{tikzpicture}
\caption{An embedding of $Q_{X(9,4)}$ into $-\mathbb{Z}^5$.}
\label{fig:embedding_A_2(1,2)}
\end{figure}
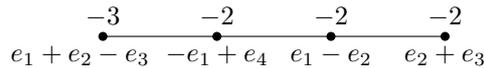

(3) Type $A_3(1,2)$: Up to an automorphism of $-\mathbb{Z}^5$, there is a unique embedding $Q_{X(18,7)}\hookrightarrow -\mathbb{Z}^5$, as shown in Figure \ref{fig:embedding_A_3(1,2)}. The orthogonal complement is generated by $e_4-e_5$, which has square $-2\neq -18$.

\begin{figure}[!th]
    \centering
\begin{tikzpicture}[scale=1]
\draw (-3,0) node[circle, fill, inner sep=1.2pt, black]{};
\draw (-1.5,0) node[circle, fill, inner sep=1.2pt, black]{};
\draw (0,0) node[circle, fill, inner sep=1.2pt, black]{};
\draw (1.5,0) node[circle, fill, inner sep=1.2pt, black]{};

\draw (-3,0) node[above]{$-3$};
\draw (-1.5,0) node[above]{$-3$};
\draw (0,0) node[above]{$-2$};
\draw (1.5,0) node[above]{$-2$};
\draw (-3,0)--(1.5,0);

\draw (-3.7,0) node[below]{$e_1+e_2-e_3$};
\draw (-1.5,0) node[below]{$-e_1+e_4+e_5$};
\draw (0.4,0) node[below]{$e_1-e_2$};
\draw (1.8,0) node[below]{$e_2+e_3$};
\end{tikzpicture}
\caption{An embedding of $Q_{X(18,7)}$ into $-\mathbb{Z}^5$.}
\label{fig:embedding_A_3(1,2)}
\end{figure}

(4) Type $A_6(1,2)$: Up to an automorphism of $-\mathbb{Z}^5$, there is a unique embedding $Q_{X(45,16)}\hookrightarrow -\mathbb{Z}^5$, as shown in Figure \ref{fig:embedding_A_6(1,2)}. The orthogonal complement is generated by $e_4-2e_5$, which has square $-5\neq -45$.

\begin{figure}[!th]
    \centering
\begin{tikzpicture}[scale=1]
\draw (-4,0) node[circle, fill, inner sep=1.2pt, black]{};
\draw (-2,0) node[circle, fill, inner sep=1.2pt, black]{};
\draw (0,0) node[circle, fill, inner sep=1.2pt, black]{};
\draw (2,0) node[circle, fill, inner sep=1.2pt, black]{};

\draw (-4,0) node[above]{$-3$};
\draw (-2,0) node[above]{$-6$};
\draw (0,0) node[above]{$-2$};
\draw (2,0) node[above]{$-2$};
\draw (-4,0)--(2,0);

\draw (-4.5,0) node[below]{$e_1+e_2-e_3$};
\draw (-2,0) node[below]{$-e_1+2e_4+e_5$};
\draw (0,0) node[below]{$e_1-e_2$};
\draw (2,0) node[below]{$e_2+e_3$};
\end{tikzpicture}
\caption{An embedding of $Q_{X(45,16)}$ into $-\mathbb{Z}^5$.}
\label{fig:embedding_A_6(1,2)}
\end{figure}

(5) Type $A_9(1,2)$: Up to an automorphism of $-\mathbb{Z}^5$, there is a unique embedding $Q_{X(72,25)}\hookrightarrow -\mathbb{Z}^5$ as shown in Figure \ref{fig:embedding_A_9(1,2)}. The orthogonal complement is generated by $e_4-e_5$, which has square $-2\neq -72$.

\begin{figure}[!th]
    \centering
\begin{tikzpicture}[scale=1]
\draw (-4,0) node[circle, fill, inner sep=1.2pt, black]{};
\draw (-2,0) node[circle, fill, inner sep=1.2pt, black]{};
\draw (0,0) node[circle, fill, inner sep=1.2pt, black]{};
\draw (2,0) node[circle, fill, inner sep=1.2pt, black]{};

\draw (-4,0) node[above]{$-3$};
\draw (-2,0) node[above]{$-9$};
\draw (0,0) node[above]{$-2$};
\draw (2,0) node[above]{$-2$};
\draw (-4,0)--(2,0);

\draw (-4.5,0) node[below]{$e_1+e_2-e_3$};
\draw (-2,0) node[below]{$-e_1+2e_4+2e_5$};
\draw (0,0) node[below]{$e_1-e_2$};
\draw (2,0) node[below]{$e_2+e_3$};
\end{tikzpicture}
\caption{An embedding of $Q_{X(72,25)}$ into $-\mathbb{Z}^5$.}
\label{fig:embedding_A_9(1,2)}
\end{figure}

(6) Type $A_{10}(1,2)$: Up to an automorphism of $-\mathbb{Z}^5$, there is a unique embedding $Q_{X(81,28)}\hookrightarrow -\mathbb{Z}^5$, as shown in Figure \ref{fig:embedding_A_10(1,2)}. The orthogonal complement is generated by $e_5$, which has square $-1\neq -81$. \end{proof}

\begin{figure}[!th]
    \centering
\begin{tikzpicture}[scale=1]
\draw (-4,0) node[circle, fill, inner sep=1.2pt, black]{};
\draw (-2,0) node[circle, fill, inner sep=1.2pt, black]{};
\draw (0,0) node[circle, fill, inner sep=1.2pt, black]{};
\draw (2,0) node[circle, fill, inner sep=1.2pt, black]{};

\draw (-4,0) node[above]{$-3$};
\draw (-2,0) node[above]{$-10$};
\draw (0,0) node[above]{$-2$};
\draw (2,0) node[above]{$-2$};
\draw (-4,0)--(2,0);

\draw (-4,0) node[below]{$e_1+e_2-e_3$};
\draw (-2,0) node[below]{$-e_1+3e_4$};
\draw (0,0) node[below]{$e_1-e_2$};
\draw (2,0) node[below]{$e_2+e_3$};
\end{tikzpicture}
\caption{An embedding of $Q_{X(81,28)}$ into $-\mathbb{Z}^5$.}
\label{fig:embedding_A_10(1,2)}
\end{figure}

The following two types are obstructed by the topological linking form condition from Lemma \ref{lem:linking_form}.
\begin{lemma} The types $A_2(1,2)E_8$ and $A_2(1,2)D_5$ do not occur as the singularity type of $S$.     
\end{lemma}
\begin{proof} We use the argument from the proof of Lemma \ref{lem:index2_linkingform}. Suppose $S$ is a $\mathbb{Q}$-homology $\mathbb{CP}^2$ with singularity type $A_2(1,2)D_5$ such that $H_1(S^0;\mathbb{Z})=0$. The associated smooth 4-manifold $Z^0$, constructed in Section \ref{subsec:top_and_smooth_obstructions}, has boundary $\partial Z^0=L(9,4)\# -L(D_5)$. The intersection form of $Z^0$ is represented by $(36)$ (Corollary \ref{cor:linking_form}), so the linking form on $H_1(\partial Z^0;\mathbb{Z})=\mathbb{Z}_{36}$ is represented by $\displaystyle\left(-\frac{1}{36}\right)$.

On the other hand, recall that $L(9,4)$ is obtained from $S^3$ by $(-{9}/{4})$-surgery along the unknot, and that $-L(D_5)$ is obtained from $S^3$ by $+4$-surgery along the right-handed trefoil knot (see Example \ref{ex:d_invs_En}). By Lemma \ref{lem:linking_form}(2), the linking forms on $H_1(L(9,4);\mathbb{Z})=\mathbb{Z}_9$ and $H_1(-L(D_5);\mathbb{Z})=\mathbb{Z}_4$ are represented by $\displaystyle\left(\frac{4}{9}\right)$ and $\displaystyle\left(-\frac{1}{4}\right)$, respectively. Via the isomorphism $H_1(L(9,4)\# -L(D_5);\mathbb{Z})=\mathbb{Z}_9\oplus \mathbb{Z}_4 \cong \mathbb{Z}_{36}$ given by $(1,1)\mapsto 1$, it follows that the linking form on $H_1(L(9,4)\# -L(D_5);\mathbb{Z})=\mathbb{Z}_{36}$ is represented by $\displaystyle\left(\frac{7}{36}\right)\cong \left(-\displaystyle\frac{29}{36}\right)$. However, since $29$ is not a quadratic residue modulo $36$, the form $\left(-\displaystyle\frac{29}{36}\right)$ is not isomorphic to $\left(-\dfrac{1}{36}\right)$, which leads to a contradiction.

The same argument can be used to show that the type $A_2(1,2)E_8$ does not occur. Note that the linking form on $L(E_8)$ is trivial, as $H_1(L(E_8);\mathbb{Z})=0$, and that $5$ is not a quadratic residue modulo $9$.
\end{proof}

Hence, the remaining types in Case $3$ are the following four:
\[A_2(1,2)E_7, \quad A_2(1,2)A_4, \quad A_2(1,2)A_1, \quad A_4(1,2)A_1. \]
\medskip

\textbf{\hypertarget{Case_4}{Case 4}:} Suppose that the unique singularity of index $3$ in $S$ is of type $A_1(2)$ or $A_n(2,2)$ ($n\geq 2$). In this case, $K_S^2=9+{8}/{3}-L>0$, so $L\leq 11$. The possible singularity type of $S$ is one of the $58$ types listed in Table \ref{tab:D_index3_case4}, and $D$ is a square number only for the following $10$ types: 
\[
A_1(2)E_8, \quad A_1(2)A_6, \quad A_1(2), \quad A_2(2,2)E_8, \quad A_2(2,2)A_9, \quad A_2(2,2)A_3,
\]  
\[
A_3(2,2)E_8, \quad A_5(2,2)A_6, \quad A_6(2,2), \quad A_{11}(2,2).
\]

\begin{table}[t]
    \centering
    \scalebox{0.9}{
    \begin{tabular}{c|c || c|c || c|c|| c|c}
      Type  &  $D$ &  Type  &  $D$ & Type  &  $D$ &  Type  &  $D$ \\ [-1em] &&&&&&&\\ 
      \hline &&&&&&&\\ [-1em]
       $A_1(2)E_8$ & $2^4$ & $A_2(2,2)A_6A_3$ & $2^3\cdot 5\cdot 7$ & $A_4(2,2)D_5$ & $2^5\cdot 11$ & $A_6(2,2)A_4$ & $5^2\cdot 17$
\\ [-1em] &&&&&&\\
       $A_1(2)A_{10}$  & $2^2\cdot 11$ & $A_2(2,2)A_6A_1$ & $2^4\cdot 5\cdot 7$ & $A_4(2,2)A_7$ & $2^4\cdot 11$  &        $A_6(2,2)A_3$ & $2^5\cdot 17$
\\[-1em] &&&&&&\\
       $A_1(2)A_6A_4$ & $2^2\cdot 5\cdot 7$ & $A_2(2,2)A_6$ & $5\cdot 7\cdot 11$ & $A_4(2,2)A_6A_1$ & $2^2\cdot 7\cdot 11$ & $A_6(2,2)A_1$ & $2^2\cdot 7\cdot 17$
\\[-1em] &&&&&&\\
       $\pmb{A_1(2)A_6}$ & $2^2\cdot 7^2$ & $A_2(2,2)A_4A_3$ & $2^5\cdot 5^2$ & $A_4(2,2)A_6$ & $5\cdot 7\cdot 11$ & $\pmb{A_6(2,2)}$ & $17^2$ 
\\[-1em] &&&&&&\\
       $A_1(2)A_4$ & $2^3\cdot 5^2$ & $A_2(2,2)A_4A_1$ & $2^2\cdot 5^2\cdot 7$ & $A_4(2,2)A_4A_3$ & $2^3\cdot 5\cdot 11$ & $A_7(2,2)A_4$ & $2^3\cdot 5^2$ 
\\[-1em] &&&&&&\\
       $\pmb{A_1(2)}$ & $2^6$ & $A_2(2,2)A_4$ & $5^2\cdot 17$ & $A_4(2,2)A_4A_1$ & $2^4\cdot 5\cdot 11$ & $A_7(2,2) $& $2^3\cdot 5\cdot 7$
\\[-1em] &&&&&&\\
       $A_2(2,2)E_8A_1$ & $2^2\cdot 5$ & $\pmb{A_2(2,2)A_3}$ & $2^4\cdot 5^2$ & $A_4(2,2)A_4$ & $5\cdot 11^2$  & $A_8(2,2)A_3$ & $2^3\cdot 23$
\\[-1em] &&&&&&\\
       $\pmb{A_2(2,2)E_8}$ & $5^2$ & $A_2(2,2)A_1$ & $2^2\cdot 5\cdot 13$ & $A_4(2,2)A_3$ & $2^3\cdot 7\cdot 11$  & $A_8(2,2)A_1$ & $2^4\cdot 23$
\\[-1em] &&&&&&\\
       $A_2(2,2)E_7$ & $2^4\cdot 5$ & $A_2(2,2)$ & $5\cdot 29$ & $A_4(2,2)A_1$ & $2^3\cdot 5\cdot 11$ & $A_8(2,2)$ & $11\cdot 23$
\\[-1em] &&&&&&\\
       $A_2(2,2)D_9$ & $2^3\cdot 5$ & $A_3(2,2)E_8$ & $2^4$ & $A_4(2,2)$ & $11\cdot 23$ & $A_9(2,2)$ & $2^4\cdot 13$
\\[-1em] &&&&&&\\
       $A_2(2,2)D_7$ & $2^5\cdot 5$ & $A_3(2,2)A_6$ & $2^6\cdot 7$ & $A_5(2,2)A_6$ & $2^2\cdot 7^2$ & $A_{10}(2,2)A_1$ & $2^2\cdot 29$
\\[-1em] &&&&&&\\
       $A_2(2,2)D_5A_4$ & $2^3\cdot 5^2$ & $A_3(2,2)A_4$ & $2^4\cdot 5\cdot 7$ & $A_5(2,2)A_4$ & $2^4\cdot 5\cdot 7$ & $A_{10}(2,2)$ & $5\cdot 29$
\\[-1em] &&&&&&\\
       $A_2(2,2)D_5$ & $2^3\cdot 5\cdot 7$ & $A_3(2,2)$ & $2^4\cdot 13$ & $A_5(2,2)$ & $2^3\cdot 5\cdot 7$ & $A_{11}(2,2)$ & $2^6$
\\[-1em] &&&&&\\
       $A_2(2,2)A_9$ & $2^2\cdot 5^2$ & $A_4(2,2)E_7$ & $2^2\cdot 11$ & $A_6(2,2)D_5$ & $2^3\cdot 17$ 
\\[-1em] &&&&&\\
       $A_2(2,2)A_7$ & $2^6\cdot 5$ & $A_4(2,2)D_7$ & $2^3\cdot 11$ & $A_6(2,2)A_4A_1$ & $2^2\cdot 5\cdot 17$ 
    \end{tabular}}
    \caption{Computation of 
    $D$ for types in \protect\hyperlink{Case_4}{Case 4}.}
    \label{tab:D_index3_case4}
\end{table}

\begin{lemma} The types $A_2(2,2)A_9$, $A_5(2,2)A_6$, and $A_{11}(2,2)$ do not occur as the singularity type of $S$.     
\end{lemma}
\begin{proof} We apply the lattice embedding condition from Corollary \ref{cor:Donaldson}, as in the proof of Lemma \ref{lem:index3_Donaldson}. The link of a singularity of type $A_n(2,2)$ ($n\geq 2$) is $L(9n-3,3n-2)$, and $-L(9n-3,3n-2)$ bounds the negative definite 4-manifold $X(9n-3,6n-1)$. The Hirzebruch-Jung continued fraction of $\displaystyle\frac{9n-3}{6n-1}$ is given by $[2,2,n+1,2,2]$.

(1) Type $A_2(2,2)A_9$: Up to an automorphism of $-\mathbb{Z}^7$, there are exactly two distinct embeddings $Q_{X(15,11)}\oplus Q_{X(10,1)}\hookrightarrow -\mathbb{Z}^7$. One such embedding, $\iota_1$, is shown in Figure \ref{fig:embedding_A_2(2,2)A_9}, and the other embedding $\iota_2$ is obtained by defining $\iota_2(v,w)=\iota_1(v,-w)$. In both cases, the orthogonal complement is generated by the vector $e_2+\cdots+e_6-e_7$, which has square $-6\neq -150$.

\begin{figure}[!th]
    \centering
\begin{tikzpicture}[scale=1]
\draw (-4,0) node[circle, fill, inner sep=1.2pt, black]{};
\draw (-2,0) node[circle, fill, inner sep=1.2pt, black]{};
\draw (0,0) node[circle, fill, inner sep=1.2pt, black]{};
\draw (2,0) node[circle, fill, inner sep=1.2pt, black]{};
\draw (4,0) node[circle, fill, inner sep=1.2pt, black]{};
\draw (0,-1.5) node[circle, fill, inner sep=1.2pt, black]{};

\draw (-4,0) node[above]{$-2$};
\draw (-2,0) node[above]{$-2$};
\draw (0,0) node[above]{$-3$};
\draw (2,0) node[above]{$-2$};
\draw (4,0) node[above]{$-2$};
\draw (0,-1.5) node[above]{$-10$};
\draw (-4,0)--(4,0);

\draw (-4,0) node[below]{$-e_2+e_3$};
\draw (-2,0) node[below]{$e_2-e_4$};
\draw (0,0) node[below]{$e_1+e_4-e_5$};
\draw (2,0) node[below]{$e_5-e_6$};
\draw (4,0) node[below]{$e_6+e_7$};
\draw (0,-1.5) node[below]{$2e_1-e_2-e_3-e_4+e_5+e_6-e_7$};
\end{tikzpicture}
\caption{An embedding of $Q_{X(15,11)}\oplus Q_{X(10,1)}$ into $-\mathbb{Z}^7$.}
\label{fig:embedding_A_2(2,2)A_9}
\end{figure}

(2) Type $A_5(2,2)A_6$: Up to an automorphism of $-\mathbb{Z}^7$, there are exactly two distinct embeddings $Q_{X(42,29)}\oplus Q_{X(7,1)}\hookrightarrow -\mathbb{Z}^7$. One such embedding, $\iota_1$, is shown in Figure \ref{fig:embedding_A_5(2,2)A_6}, and the other embedding, $\iota_2$, is obtained by defining $\iota_2(v,w)=\iota_1(v,-w)$. In both cases, the orthogonal complement is generated by the vector $e_2+e_3-e_4+e_5+e_6-e_7$, which has square $-6\neq -294$. 
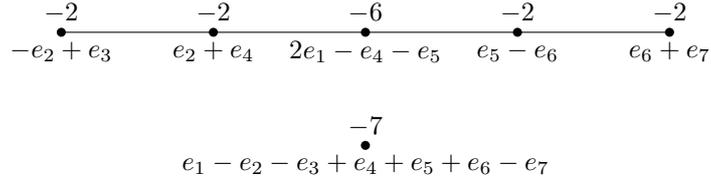
\begin{figure}[!th]
    \centering
\begin{tikzpicture}[scale=1]
\draw (-4,0) node[circle, fill, inner sep=1.2pt, black]{};
\draw (-2,0) node[circle, fill, inner sep=1.2pt, black]{};
\draw (0,0) node[circle, fill, inner sep=1.2pt, black]{};
\draw (2,0) node[circle, fill, inner sep=1.2pt, black]{};
\draw (4,0) node[circle, fill, inner sep=1.2pt, black]{};
\draw (0,-1.5) node[circle, fill, inner sep=1.2pt, black]{};

\draw (-4,0) node[above]{$-2$};
\draw (-2,0) node[above]{$-2$};
\draw (0,0) node[above]{$-6$};
\draw (2,0) node[above]{$-2$};
\draw (4,0) node[above]{$-2$};
\draw (0,-1.5) node[above]{$-7$};
\draw (-4,0)--(4,0);

\draw (-4,0) node[below]{$-e_2+e_3$};
\draw (-2,0) node[below]{$e_2+e_4$};
\draw (0,0) node[below]{$2e_1-e_4-e_5$};
\draw (2,0) node[below]{$e_5-e_6$};
\draw (4,0) node[below]{$e_6+e_7$};
\draw (0,-1.5) node[below]{$e_1-e_2-e_3+e_4+e_5+e_6-e_7$};
\end{tikzpicture}
\caption{An embedding of $Q_{X(42,29)}\oplus Q_{X(7,1)}$ into $-\mathbb{Z}^7$.}
\label{fig:embedding_A_5(2,2)A_6}
\end{figure}

(3) Type $A_{11}(2,2)$: Up to an automorphism of $-\mathbb{Z}^6$, there is a unique embedding $Q_{X(96,65)}\hookrightarrow -\mathbb{Z}^6$, as shown in Figure \ref{fig:embedding_A_11(2,2)}. The orthogonal complement is generated by $e_1+e_2-e_3+e_4+e_5-e_6$, which has square $-6\neq -96$. \end{proof}

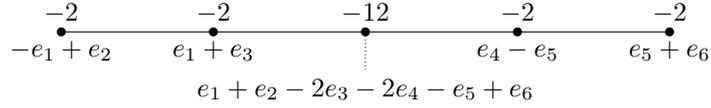
\begin{figure}[!th]
    \centering
\begin{tikzpicture}[scale=1]
\draw (-4,0) node[circle, fill, inner sep=1.2pt, black]{};
\draw (-2,0) node[circle, fill, inner sep=1.2pt, black]{};
\draw (0,0) node[circle, fill, inner sep=1.2pt, black]{};
\draw (2,0) node[circle, fill, inner sep=1.2pt, black]{};
\draw (4,0) node[circle, fill, inner sep=1.2pt, black]{};

\draw (-4,0) node[above]{$-2$};
\draw (-2,0) node[above]{$-2$};
\draw (0,0) node[above]{$-12$};
\draw (2,0) node[above]{$-2$};
\draw (4,0) node[above]{$-2$};
\draw (-4,0)--(4,0);
\draw[densely dotted] (0,0)--(0,-0.5);

\draw (-4,0) node[below]{$-e_1+e_2$};
\draw (-2,0) node[below]{$e_1+e_3$};
\draw (0,-0.5) node[below]{$e_1+e_2-2e_3-2e_4-e_5+e_6$};
\draw (2,0) node[below]{$e_4-e_5$};
\draw (4,0) node[below]{$e_5+e_6$};
\end{tikzpicture}
\caption{An embedding of $Q_{X(96,65)}$ into $-\mathbb{Z}^6$.}
\label{fig:embedding_A_11(2,2)}
\end{figure}

\begin{lemma} The types $A_1(2)E_8$ and $A_3(2,2)E_8$ do not occur as the singularity type of $S$.     
\end{lemma}
\begin{proof} The link $L(6,1)$ of a singularity of type $A_1(2)$ has two spin structures with $d$-invariants $-{5}/{4}$ and ${1}/{4}$, and the link $L(24,7)$ of a singularity of type $A_3(2,2)$ has two spin structures with $d$-invariants $-{5}/{4}$ and ${3}/{4}$ (see formula (\ref{eq:d-invariant})). On the other hand, the link $L(E_8)$ of a singularity of type $E_8$ has a unique spin structure with $d$-invariant of $2$ (Table \ref{tab:d_invs}). Therefore, the condition from Corollary \ref{cor:spin_d_2} is not satisfied in either case. 
\end{proof}

Hence, the remaining types in Case $4$ are the following five: 
\[A_1(2)A_6, \quad A_1(2), \quad A_2(2,2)E_8, \quad A_2(2,2)A_3, \quad A_6(2,2).\]
\medskip

\textbf{Case 5:} Suppose that the unique singularity of index 3 in $S$ is of type $D_n(1)$ $(n\geq 4)$. For even $n$, the first homology group $H_1(L(D_n(1));\mathbb{Z})$ of the link $L(D_n(1))$ of a singularity of type $D_n(1)$ is $\mathbb{Z}_6\oplus \mathbb{Z}_2$ \cite[Satz 2.11]{Brieskorn-1968}, which is not cyclic. Thus, by Corollary \ref{cor:linking_form}, $S$ cannot have a singularity of type $D_n(1)$ for even $n$. In particular, we may assume $n\geq 5$, in which case $K_S^2=9+\frac{2}{3}-L>0$ and $L\leq 9$. Since the numbers $|\det(R_p)|$ are pairwise relatively prime, it follows that the singularity type of $S$ is one of the following four: $D_5(1)A_4$, $D_5(1)$, $D_7(1)$, and $D_9(1)$. However, $D$ is equal to $2^3\cdot 5$, $2^3\cdot 7$, $2^5$, and $2^3$, respectively, meaning that $D$ is a square number for any of these types. Hence, Case $5$ does not occur.
\medskip

\textbf{Case 6:} Suppose that the unique singularity of index $3$ in $S$ is of type $D_n(2)$ $(n\geq 4)$. For even $n$, the first homology group $H_1(L(D_n(2));\mathbb{Z})=\mathbb{Z}_6\oplus \mathbb{Z}_2$ is not cyclic, so $S$ cannot have a singularity of type $D_n(2)$ for even $n$. We have $K_S^2=9+\frac{4}{3}-L>0$, so $L\leq 10$. The possible singularity type of $S$ is one of the following four: $D_5(2)A_4$, $D_5(2)$, $D_7(2)$, and $D_9(2)$. However, $D$ is equal to $2^4\cdot 5$, $2^6$, $2^3\cdot 5$, and $2^4$, respectively. Thus, $D$ is a square number only for the two types $D_5(2)$ and $D_9(2)$.

\begin{lemma} The type $D_9(2)$ does not occur as the singularity type of $S$.     
\end{lemma}
\begin{proof} Consider the link $L(D_9(2))$. Note that $H_1(L(D_9(2));\mathbb{Z})=\mathbb{Z}_{12}$, and in particular, $L(D_9(2))$ has exactly two distinct spin structures. Following the notation in \cite[p.672]{Doig-2015}, the orientation reversal $-L(D_9(2))$ is a Seifert manifold $\left(-1;\frac{1}{2},\frac{1}{2},\frac{3}{19}\right)$. For the two spin structures of $-L(D_9(2))$, the corresponding $d$-invariants are $-\frac{5}{4}$ and $-\frac{9}{4}$ \cite[Table 3]{Doig-2015}. Therefore, the condition from Corollary \ref{cor:spin_d_2} cannot be satisfied.    
\end{proof}

Hence, $D_5(2)$ is the only remaining type in Case 6.
\medskip

In conclusion, the singularity types that can occur for a $\mathbb{Q}$-homology $\mathbb{CP}^2$ of index $3$ with trivial first integral homology group are one of 18 types: 6 types containing $A_1(1)$, 2 types containing $A_n(1,1)$, 4 types containing $A_n(1,2)$, 5 types containing $A_1(2)$ or $A_n(2,2)$, and 1 type containing $D_n(2)$, as stated in Theorem \ref{thm:index3_list}.

\subsection{Realizations}\label{subsec:realization}
Here we provide realizations for the 18 types in Theorem \ref{thm:index3_list}, except for the two types $A_2(1,2)E_7$ and $A_2(2,2)E_8$. As mentioned earlier, the realizations of the $6$ types $A_1(1)E_8$, $A_1(1)D_7$, $A_1(1)A_6$, $A_1(1)A_4A_1$, $A_1(1)A_3$, and $A_1(1)$ are given in \cite[Appendix]{Zhang-1989}. The remaining 10 types are: 

\[A_6(1,1),\quad A_{10}(1,1),\quad A_2(1,2)A_4,\quad A_2(1,2)A_1,\quad A_4(1,2)A_1,\quad A_1(2)A_6,\] 
\[A_1(2),\quad A_2(2,2)A_3,\quad A_6(2,2),\quad D_5(2).\]
For each of these 10 types, the configuration of the exceptional locus $D$ in the minimal resolution $\tilde{S}$ is given in Table \ref{tab:configuration}. In each case, the minimal resolution is obtained by a sequence of blow-ups from a Hirzebruch surface.

\begin{table}[t]
    \centering
    \begin{tabular}{cc|cc}

\begin{tikzpicture}[scale=1]
\draw (-2,0)--(2,0) (1.5,-0.2)--(1.5,1.2) (1.7,1)--(0.3,1) (0.5,0.8)--(0.5,2.2) (1.7,2)--(0.3,2) (1.5,1.8)--(1.5,3);
\draw[densely dashed] (0.7,1.5)--(-0.7,1.5);
\draw (1,2) node[above]{$-3$};
\draw (1.5,0.5) node[right]{$-3$};
\draw (-1.5,0) node{$*$};
\end{tikzpicture}

& & &

\begin{tikzpicture}[scale=1.1]
\draw (-3.5,0)--(2,0)  (1.7,1)--(0.3,1) (0.5,0.8)--(0.5,2.2) (2.2,2)--(0.3,2) (1.5,1.8)--(1.5,3.2) (-3,-0.2)--(-3,1.2) (-3,1.8)--(-3,3.2) (-3.5,3)--(2,3) (-3.2,0.8)--(-1.8,1.7) (-3.2,2.2)--(-1.8,1.3);
\draw[densely dashed] (1.5,-0.2)--(1.5,1.2) (2,1.5)--(2,2.5) (-2.7,1.5)--(-1.8,0.5);
\draw (0.5,1.5) node[left]{$-3$};
\draw (-3,0.4) node[left]{$-3$};
\draw (-1,0) node{$*$};
\draw (-1,3) node{$*$};

\draw (0,0) node[above]{$D_1$};
\draw (-3,0.2) node[right]{$D_2$};
\draw (-2.9,0.9) node[right]{$D_3$};
\draw (-2.6,1.9) node[right]{$D_4$};
\draw (-3,2.6) node[right]{$D_5$};
\draw (-1.5,3) node[below]{$D_6$};
\draw (1.5,2.7) node[left]{$D_7$};
\draw (0.9,2) node[above]{$D_8$};
\draw (0.5,1.5) node[right]{$D_9$};
\draw (1,1) node[below]{$D_{10}$};
\draw (1.5,0.5) node[right]{$E_1$};
\draw (2,1.5) node[below]{$E_2$};
\draw (-1.9,0.5) node[right]{$E_3$};
\end{tikzpicture}

\\

$A_6(1,1)$ && & $A_{10}(1,1)$ 
\\ [-1em] &&& \\  
$\tilde{S}=\Sigma_2\#5\overline{\mathbb{CP}}^2$ & & & $\tilde{S}=\Sigma_2\#9\overline{\mathbb{CP}}^2$ \\
      \hline & & & \\ [-1em] &&& \\  

\begin{tikzpicture}[scale=0.8]
\draw (-2,0)--(2,0) (1.5,-0.2)--(1.5,1.2)  (0.5,0.8)--(0.5,2.2) (1.7,2)--(0.3,2) (1.5,1.8)--(1.5,3.2) (1.7,3)--(0.4,3);
\draw[densely dashed] (1.7,1)--(0.3,1);
\draw (1.5,0.5) node[left]{$-5$};
\draw (-1.5,0) node{$*$};
\end{tikzpicture}

&& &

\begin{tikzpicture}[scale=0.8]
\draw (-2,0)--(2,0) (1.5,-0.2)--(1.5,1.2)  (0.5,0.8)--(0.5,2.2) ;
\draw[densely dashed] (1.7,1)--(0.3,1);
\draw (-1,0) node[above]{$-5$};
\draw (-1.5,0) node{$*$};
\end{tikzpicture}

\\
$A_2(1,2)A_4$ && & $A_{2}(1,2)A_1$
\\ [-1em] &&& \\  
$\tilde{S}=\Sigma_2\#5\overline{\mathbb{CP}}^2$ & & & $\tilde{S}=\Sigma_5\#2\overline{\mathbb{CP}}^2$ \\
      \hline & & & \\ [-1em] &&& \\ 

\begin{tikzpicture}[scale=1]
\draw (-2,0)--(2,0) (1.5,-0.2)--(1.5,1.2) (1.7,1)--(-0.5,1) (0.5,0.8)--(0.5,2.2) (-0.1,2)--(-1.2,2);
\draw[densely dashed] (-0.3,0.8)--(-0.3,2.2);
\draw (1,1) node[below]{$-3$};
\draw (-1,0) node[above]{$-4$};
\draw (-1.5,0) node{$*$};
\end{tikzpicture}

&& &

\begin{tikzpicture}[scale=1]
\draw (-2,0.6)--(2,0.6) (1.5,0.5)--(1.5,1.4) (1.7,1.3)--(0.3,1.3) (0.5,1.2)--(0.5,2.1) (1.7,2)--(0.3,2) (1.5,1.9)--(1.5,2.6) (0.5,2.4)--(0.5,3);
\draw[densely dashed] (1.7,2.5)--(0.4,2.5);
\draw (0.5,2.7) node[left]{$-6$};
\draw (-1.5,0.6) node{$*$};
\end{tikzpicture}

\\
$A_4(1,2)A_1$ && & $A_{1}(2)A_6$
\\ [-1em] &&& \\  
$\tilde{S}=\Sigma_4\#4\overline{\mathbb{CP}}^2$ & & & $\tilde{S}=\Sigma_2\#6\overline{\mathbb{CP}}^2$ \\ 
      \hline & & & \\ [-1em] &&& \\ 

\begin{tikzpicture}[scale=0.8]
\draw (-2,0)--(2,0) ;
\draw (-1,0) node[above]{$-6$};
\draw (-1.5,0) node{$*$};
\end{tikzpicture}

& & &

\begin{tikzpicture}[scale=0.8]
\draw (-2,0)--(2,0) (1.5,-0.2)--(1.5,1.2)  (0.5,0.8)--(0.5,2.2) (1.7,2)--(0.3,2) (1.5,1.8)--(1.5,3);
\draw[densely dashed] (1.7,1)--(0.3,1);
\draw (1.5,0.5) node[right]{$-4$};
\draw (-1,0) node[above]{$-4$};
\draw (-1.5,0) node{$*$};
\end{tikzpicture}

\\
$A_1(2)$ && & $A_{2}(2,2)A_3$
\\ [-1em] &&& \\  
$\tilde{S}=\Sigma_6$ & & & $\tilde{S}=\Sigma_4\#4\overline{\mathbb{CP}}^2$ \\
      \hline & & & \\ [-1em] &&& \\ 

\begin{tikzpicture}[scale=0.8]
\draw (-1.5,0)--(2.5,0) (1.5,-0.2)--(1.5,1.2) (1.7,1)--(0.3,1) (0.5,0.8)--(0.5,2.2) (1.7,2)--(0.3,2) (1.5,1.8)--(1.5,3);
\draw[densely dashed] (1,2.5)--(1,1.5);
\draw (-0.5,0) node[above]{$-4$};
\draw (1.5,2.5) node[right]{$-4$};
\draw (-1,0) node{$*$};
\end{tikzpicture}

&& &

\begin{tikzpicture}[scale=0.8]
\draw (-2,0)--(2,0) (1.5,-0.2)--(1.5,1.2) (1.7,1)--(-0.5,1) (0.6,0.8)--(0.6,2.2)  (-0.3,0.8)--(-0.3,2.2);
\draw[densely dashed] (-0.1,2)--(-1.2,2);
\draw (-1,0) node[above]{$-4$};
\draw (-1.5,0) node{$*$};
\end{tikzpicture} \\
$A_6(2,2)$ && & $D_5(2)$ \\ [-1em] &&& \\  
$\tilde{S}=\Sigma_4\#5\overline{\mathbb{CP}}^2$ & & & $\tilde{S}=\Sigma_4\#4\overline{\mathbb{CP}}^2$ 
\end{tabular}
\caption{Configurations of the exceptional locus in the minimal resolution}   \label{tab:configuration}
\end{table}

For the configurations, we follow the convention in \cite{Zhang-1989}. A solid line represents a component of $D$. The self-intersection number of a $(-2)$-curve in $D$ is omitted. A dotted line in the configuration represents a $(-1)$-curve. A line with $*$ indicates that the corresponding curve is not contained in any fiber of the vertical $\mathbb{CP}^1$-fibration on the minimal resolution $\tilde{S}$. Additionally, $\Sigma_n$ denotes the Hirzebruch surface of degree $n$.

Now we show that, in each case, the smooth locus $S^0$ of the given surface $S$ is simply-connected. Except for the type $A_{10}(1,1)$, it follows directly from the configurations that, in each case, the smooth locus $S^0$ contains $\Sigma_n\setminus ((\textup{fiber}) \cup (\textup{section}))\cong \mathbb{C}^2$ as a Zariski open subset. Hence $S^0$ is simply-connected by a transversality argument.

Next, we handle the type $A_{10}(1,1)$. We first show that $H_1(S^0;\mathbb{Z})=0$. Since $b_1(S^0)=0$, it suffices to show that $H_1(S^0;\mathbb{Z})$ is torsion-free, or equivalently, that $H^2(S^0;\mathbb{Z})$ is torsion-free. Note that there is a short exact sequence \cite[Lemma 2(2)]{Miyanishi-Zhang-1988} \[
0\to H_2(D;\mathbb{Z})\to \textup{Pic}(\tilde{S})\to H^2(S^0;\mathbb{Z})\to 0.
\]

Let $\Phi\colon \tilde{S}\to \mathbb{CP}^1$ denote the vertical $\mathbb{CP}^1$-fibration (defined by the linear pencil $|D_5+D_4+2(D_3+E_3)+D_4+D_5|$) on the minimal resolution given in the configuration. Letting $F$ be a fiber of $\Phi$, we have linear equivalences \[
F\sim D_5+D_4+2(D_3+E_3)+D_4+D_5\sim D_7+2(D_8+E_2)+D_9+D_{10}+E_1
\]
and \[
D_6\sim D_1+2F-(D_5+D_4+D_3+E_3)-(D_{10}+2D_9+3D_7+5D_8+5E_2).
\]
Note that $\textup{Pic}(\tilde{S})(\cong H^2(\tilde{S};\mathbb{Z})\cong H_2(\tilde{S};\mathbb{Z}))$ is a free abelian group of rank $11$, with the following basis: \[
\{D_1,D_3,D_4,D_5,D_7,D_8,D_9,D_{10},F,E_2,E_3\}.
\]
In $\textup{Pic}(\tilde{S})/\langle D_1,\dots,D_{10}\rangle$, we have $F=2E_3$ and $2F=E_3+5E_2$. Hence, $H^2(S^0;\mathbb{Z})\cong \mathbb{Z}$, and it follows that $H_1(S^0;\mathbb{Z})=0$ (cf. the proof of \cite[Proposition 4.13]{Kojima-1999}).

Now we show that $\pi_1(S^0)=1$ using the argument in \cite{Zhang-1989}. Let $\sigma\colon\tilde{S}_1\to \tilde{S}$ be the blow-up at the point $P:=D_5\cap D_6$. Let $E_4:=\sigma^{-1}(P)$, and let $\tilde{D}_5$ and $\tilde{D}_6$ denote the proper transforms of $D_5$, $D_6$, respectively. Set \[F_0=D_1+\tilde{D}_5+2D_2+3D_4+5(D_3+E_3)\quad\text{and}\quad F_1=\tilde{D}_6+D_{10}+2D_9+3D_7+5(D_8+E_2).\] Then, $|F_0|$ defines a $\mathbb{CP}^1$-fibration $\varphi\colon\tilde{S}_1\to \mathbb{CP}^1$ with precisely two singular fibers, $F_0$ and $F_1$. All components of $\sigma^{-1}(D)$, except $E_4$, are contained in the singular fibers of $\varphi$. Note that $E_1$ and $E_4$ are sections of $\varphi$. 

Let $\tau\colon\tilde{S}_1\to \Sigma_0$ be the contraction of the curves in $F_0$ and $F_1$ except $D_1$ and $\tilde{D}_6$. Then $\tau(\sigma^{-1}(D))$ is the union $\tau(F_0)\cup \tau(F_1)\cup \tau(E_4)$. Therefore, \[
\tilde{S}\setminus (D\cup E_2\cup E_3)=\Sigma_0\setminus \tau(\sigma^{-1}(D))\cong \mathbb{C}\times \mathbb{C}^*,\]
where $\mathbb{C}^*=\mathbb{C}\setminus \{0\}$. That is, $S^0$ contains $\mathbb{C}\times \mathbb{C}^*$ as a Zariski open subset, implying that $\pi_1(S^0)$ is a quotient group of $\mathbb{Z}$. Since we already know that $H_1(S^0;\mathbb{Z})=0$, we conclude that $\pi_1(S^0)=1$.

\begin{remark} Note that, in each case except for the two types $A_{10}(1,1)$ and $D_5(2)$, one can also establish the simple connectivity of the smooth locus using the topological argument provided in the proof of \cite[Theorem 3]{Lee-Park-2007}. \end{remark}

\clearpage

\bibliography{references}{}
\bibliographystyle{alpha}
\end{document}